\documentclass{amsart}

\usepackage{amsmath}
\usepackage{amssymb}
\usepackage{amsthm}
\usepackage{color}
\usepackage{mathrsfs}
\usepackage{url}
\usepackage{enumitem}
\usepackage{stmaryrd}
\usepackage[all]{xy}
\usepackage{verbatim}

\newtheorem{Thm}{Theorem}[section]
\newtheorem{Lem}[Thm]{Lemma}
\newtheorem{Prop}[Thm]{Proposition}

\newtheorem{Cor}[Thm]{Corollary}

\theoremstyle{definition}

\newtheorem{Eg}[Thm]{Example}
\newtheorem{Def}[Thm]{Definition}

\newtheorem{Hyp}[Thm]{Hypothesis}

\newtheorem{Rmk}[Thm]{Remark}

\numberwithin{equation}{Thm}

\newcommand{\CC}{\mathbb{C}}
\newcommand{\Cp}{\mathbb{C}_p}
\newcommand{\FF}{\mathbb{F}}

\newcommand{\NN}{\mathbb{N}}
\newcommand{\OCp}{\mathcal{O}_{\Cp}}
\newcommand{\Qab}{\mathbb{Q}^{\text{ab}}}

\newcommand{\QQ}{\mathbb{Q}}

\newcommand{\ZZ}{\mathbb{Z}}

\newcommand{\Ptyp}{P}

\newcommand{\Warrow}{\underleftarrow{W}}

\newcommand{\WWarrow}{\underleftarrow{\mathbb{W}}}

\DeclareMathOperator{\Ab}{\mathbf{Ab}}

\DeclareMathOperator{\Gal}{Gal}
\DeclareMathOperator{\Hom}{Hom}
\DeclareMathOperator{\image}{image}

\DeclareMathOperator{\Ring}{\mathbf{Ring}}
\DeclareMathOperator{\Set}{\mathbf{Set}}

\DeclareMathOperator{\Trace}{Trace}

\author{Christopher Davis}
\address{University of California, Irvine, Dept of
Mathematics, Irvine, CA 92697}
\curraddr{University of Copenhagen, Dept of Mathematical Sciences,  Universitetsparken~5, 
2100 K{\o}benhavn {\O}, Denmark}
\email{davis@math.ku.dk}
\date{\today}

\author{Kiran S. Kedlaya}
\address{University of California, San Diego, Dept of Mathematics, La Jolla, CA 92093}
\email{kedlaya@ucsd.edu}

\thanks{Davis was supported by the Danish National Research Foundation through the Centre for Symmetry and Deformation (DNRF92).  Kedlaya was supported by NSF (grant DMS-1101343) and
UCSD (Stefan E. Warschawski professorship).}

\CompileMatrices

\begin{document}
\title{Almost purity and overconvergent Witt vectors}

\begin{abstract}
In a previous paper, we stated a general \emph{almost purity theorem} in the style of Faltings: if $R$ is a ring for which the Frobenius maps on finite $p$-typical Witt vectors over $R$ are surjective, then the integral closure of $R$ in a finite \'etale extension of $R[p^{-1}]$ 
is ``almost'' finite \'etale over $R$. Here, we use almost purity to
lift the finite \'etale extension of $R[p^{-1}]$ to a finite \'etale extension of rings of \emph{overconvergent Witt vectors}. The point is that no hypothesis of $p$-adic completeness is needed; this result thus points towards potential global analogues of $p$-adic Hodge theory. As an illustration,
we construct $(\varphi, \Gamma)$-modules associated to Artin Motives over $\QQ$.  The
$(\varphi, \Gamma)$-modules we construct are defined over a base ring which seems
well-suited to generalization to a more global setting; we plan to pursue
such generalizations in later work.
\end{abstract}

\maketitle

\section*{Introduction}

This paper builds on our previous paper \cite{DK12} to take an initial step towards constructing a global analogue of $p$-adic Hodge theory. To explain what this assertion means, let us first recall some aspects of the usual $p$-adic theory of $(\varphi, \Gamma)$-modules (see \cite[Section~13.6]{BC09} for a detailed description).  Let $K$ denote a finite extension of $\QQ_p$ and put $K_{\infty} := K(\mu_{p^{\infty}})$.  Let $G_K$ and $G_{K_{\infty}}$ denote the absolute Galois groups of $K$ and $K_{\infty}$.  Put $\Gamma := \Gal(K_{\infty}/K)$; it is isomorphic to an open subgroup of $\ZZ_p^{\times}$.   Let $V$ denote a finite-dimensional $p$-adic representation of $G_{K}$.  The theory of $(\varphi, \Gamma)$-modules concerns replacing 
\begin{itemize}
\item the $p$-adic representation $V$, which is a finite module over a nice ring (in our case, $\QQ_p$) equipped with an action by a complicated group (in our case, $G_K$)
\end{itemize}
with
\begin{itemize}
\item a finite module over a complicated ring equipped with an action by a single operator $\varphi$ and the group $\Gamma$.
\end{itemize}
More precisely, for a certain choice of $\QQ_p$-algebra $\mathbf{A}$, the functor from $G_{K_{\infty}}$-representations to $(\varphi, \Gamma)$-modules is given by 
\[
V \leadsto (\mathbf{A} \otimes_{\QQ_p} V)^{G_{K_{\infty}}}.
\]
This gives a $(\varphi, \Gamma)$-module over  $\mathbf{A}_K := \mathbf{A}^{G_{K_{\infty}}}$.

Within this framework, there are multiple reasonable choices for the rings $\mathbf{A}$ and $\mathbf{A}_K$. However, the constructions of these rings in general depend rather heavily on the prime $p$.
Our overarching goal, of which this paper represents a small part, is to define a base ring which serves a similar purpose, but where the dependence on $p$ is concentrated in the choice of a norm (corresponding to the $p$-adic norm).  This provides hope for the possibility of globalizing the construction in a way that allows us to treat not only all primes simultaneously, but perhaps also the archimedean place as well.

Evidence that there might be a meaningful global theory of $(\varphi, \Gamma)$-modules is provided in \cite{Ked13b}, which uses a special choice of base ring $\mathbf{A}_K$ to keep track of a rational structure on de\thinspace Rham cohomology.  On the other hand, it is not clear how to adapt that construction to work for multiple primes $p$ simultaneously.  

In this paper, we describe a construction of \emph{overconvergent Witt vectors} which we propose as a candidate ingredient for a global theory of $(\varphi, \Gamma)$-modules. Our overconvergent Witt vectors are inverse limits of finite $p$-typical Witt vectors along Frobenius subject to a growth condition; when evaluated at a ring of characteristic $p$, we recover the definition used by Davis--Langer--Zink to describe rigid cohomology in terms of the de Rham-Witt complex \cite{DLZ11b, DLZ11a}. However, our present definition makes sense also for rings equipped with $p$-adic norms. For a ring $R$ complete for such a norm, there is an isomorphism between the construction over $R$ and the corresponding construction over an associated ring in characteristic $p$ (the \emph{tilt} in the sense of Scholze \cite{Sch11}, which generalizes the Fontaine-Wintenberger \emph{field of norms} construction). For a more general ring equipped with a $p$-adic norm, the overconvergent Witt vectors form a sort of \emph{decompletion} of the corresponding ring over the $p$-adic completion.

The principal result of this paper (Theorem~\ref{almost purity part 2})
asserts a compatibility between the formation of overconvergent Witt vectors and the formation of finite \'etale extensions. We view this as a variant of the \emph{almost purity theorem} stated in our previous paper \cite{DK12} (and based on work of Faltings \cite{Fal88},
Kedlaya--Liu \cite{KL11}, and Scholze \cite{Sch11}). That theorem asserts that under a certain hypothesis on a ring $R$, the integral closure of $R$ in a finite \'etale extension of $R[p^{-1}]$ is \emph{almost finite \'etale} over $R$; roughly speaking, this means that the failure of this extension to be truly finite \'etale is negligible compared to any ideal of definition of the $p$-adic topology. The key condition on $R$ is that it be \emph{Witt-perfect}, meaning that the Frobenius maps on finite $p$-typical Witt vectors over $R$ are surjective; this is closely related to the defining condition of \emph{perfectoid algebras} used in \cite{KL11} and \cite{Sch11}.

To illustrate the intent of our work, in Section~\ref{sec:phi,Gamma} we
use overconvergent Witt vectors to give an explicit construction analogous to the $(\varphi, \Gamma)$-module construction in the special case of \emph{Artin motives}, i.e., under the special assumption that our action of $G_K$ on $V$ factors through a finite subgroup.  This is done in Theorem~\ref{Artin motive theorem}.  The main result required by our $(\varphi, \Gamma)$-module construction is Theorem~\ref{almost purity part 2}.

It is worth comparing this construction with that of \cite{Ked13b}. On one hand, the construction there uses a ring even smaller than the one considered here, and is able to handle motives other than Artin motives. On the other hand, our construction is better set up for globalization by replacing $p$-typical Witt vectors with big Witt vectors; it is also formulated so as to allow replacement of $p$-adic norms with other norms, possibly even archimedean norms. A pressing question for future work is to find a natural way to combine the constructions so as to retain the advantages on both sides.

\subsection*{Acknowledgements} 

The authors thank Laurent Berger, James Borger, Bryden Cais, Lars Hesselholt, Abhinav Kumar, 
Ruochuan Liu, Tommy Occhipinti, Joe Rabinoff, Daqing Wan, and Liang Xiao for helpful discussions and
suggestions.  

\section{Background on Witt vectors}

\begin{Hyp}
By convention, all rings considered in this paper are commutative and unital, and ring homomorphisms send $1$ to $1$.  When we say a ring $R$ is $p$-adically complete, we implicitly include the assumption that $R$ is also $p$-adically separated.
\end{Hyp}

\begin{Def}
Let $\NN$ denote the multiplicative monoid of positive integers. For $p$ a prime number fixed throughout the paper, let $\Ptyp \subset \NN$ denote the multiplicative monoid of powers of $p$, i.e.,  $\Ptyp = \{1,p,p^2,\ldots\}$.  
Let $\Set, \Ab, \Ring$ denote the categories of sets, abelian groups, and rings, respectively. 

Let $\bullet^{\Ptyp}: \Ring \to \Ring$ denote the functor $R \mapsto R^\Ptyp$ for the product ring structure on $R^\Ptyp$; we will use the same notation to denote the underlying functors $\Ring \to \Ab$ and $\Ring \to \Set$.
Let $F: \bullet^{\Ptyp} \to \bullet^{\Ptyp}$ be the natural transformation of functors
$\Ring \to \Ring$ given by left shift: 
\[
(x_1, x_p, \dots) \mapsto (x_p, x_{p^2}, \dots).
\]
Let $V: \bullet^{\Ptyp} \to \bullet^{\Ptyp}$ be the natural transformation of functors
$\Ring \to \Ab$ given by
\[
(x_1, x_p, x_{p^2}\dots) \mapsto (0, px_1, px_p, \dots).
\]
\end{Def}

% I wasn't sure what exponent to use for $p$-powers.  Right now I'm using just a capital P, but that can be changed by changing what $\Ptyp$ means.  -CJD

\begin{Def}
Define the functor $W_0: \Ring \to \Set$ by the formula $R \mapsto R^{\Ptyp}$.  Of course, this is the same as the functor $\bullet^{\Ptyp}$ viewed as a functor $\Ring \to \Set$.  Define the \emph{ghost map} as the natural transformation $w: W_0 \to \bullet^{\Ptyp}$ by the formula
\[
(x_1, x_p, \dots) \mapsto (w_1, w_p, \dots), \qquad w_{p^n} = \sum_{0 \leq i \leq n} p^i x_{p^i}^{p^{n-i}}.
\]
\end{Def}

\begin{Thm} \label{construct Witt vectors}
\begin{enumerate}
\item[(a)]
There is a unique way to promote $W_0$ to a functor $W: \Ring \to \Ring$ in such a way that
$w: W \to \bullet^{\Ptyp}$ is a natural transformation.
\item[(b)]
There is a unique natural transformation $F: W \to W$
such that $w \circ F = F \circ w$.
\item[(c)]
The formula 
\[
(x_1, x_p, x_{p^2}, \dots) \mapsto (0, x_1, x_p, \dots)
\]
defines a natural transformation $V: W \to W$ of functors $\Ring \to \Ab$ satisfying $w \circ V = V \circ w$.
\end{enumerate}
\end{Thm}

\begin{proof}
For part (a), see \cite[Chapter~II, Theorem~6]{Serre79}.  For parts (b) and (c), see \cite[Sections~0.1.1--0.1.3]{Ill79}. 
\end{proof}

We will be defining norms and norm-like functions on $p$-typical Witt vectors over normed rings.  To prove that these functions really behave like norms, we will have to study the effect of addition and multiplication in terms of components.  The most convenient way to study these Witt vector components is via ghost components.

\begin{Lem} \label{homogeneous polys}
Let $R = \ZZ[x_1, x_p, \ldots]$ and fix $a \in \NN$.  Assume $\underline{z} \in W(R)$ is such that for all $n$, the $p^n$-th ghost component of $\underline{z}$ is a homogeneous polynomial of degree $ap^n$ under the weighting in which the variable $x_{p^i}$ has weight $p^i$.  Then the $p^n$-th Witt component of $\underline{z}$, written $z_{p^n}$, is also a homogeneous polynomial of degree $ap^n$.  
\end{Lem}

\begin{proof}
We prove this using induction on $n$.  If $n = 0$, the result is trivial, because $w_{p^0} = z_{p^0}$.  To complete the induction, note that
\[
p^{n+1}z_{p^{n+1}} = w_{p^{n+1}} - \sum_{0 \leq i \leq n} p^i z_{p^i}^{p^{n+1-i}}.
\]  
Everything on the right side is homogeneous of degree $ap^{n+1}$, and hence so is $z_{p^{n+1}}$.
\end{proof}

The following corollary is an immediate application of Lemma~\ref{homogeneous polys}.

\begin{Cor} \label{homogeneous corollary} Let $R$ denote a ring, and let $\underline{x}, \underline{y} \in W(R)$.  
\begin{enumerate}
\item[(a)] The $p^i$-th component of $\underline{x} + \underline{y}$ is a homogeneous polynomial of degree $p^i$ in $x_1, x_p, \ldots, y_1, y_p, \ldots$, under the weighting in which the variables $x_{p^i}$, $y_{p^i}$ have weight $p^i$.
\item[(b)] The $p^i$-th component of $\underline{x} \cdot \underline{y}$ is a homogeneous polynomial of degree $2p^i$.
\item[(b$'$)] The $p^i$-th component of $\underline{x} \cdot \underline{y}$ is a homogeneous polynomial in the $x$-variables (respectively, in the $y$-variables) of degree $p^i$ (respectively, of degree $p^i$).
\item[(c)] The $p^i$-th component of $F(\underline{x})$ is a homogeneous polynomial of degree $p^{i+1}$.
\end{enumerate}
\end{Cor}

In \cite{DK12}, the Witt-vector Frobenius map $F$ was studied extensively.  We recall now some of the most important concepts and results from that paper.
We begin with a more precise description of the components of Frobenius.  For the proof, see \cite[Lemma~1.4(a)]{DK12}.

\begin{Lem} \label{big Frobenius components lemma}
Take $\underline{x}, \underline{y} \in W(R)$ with $F(\underline{x}) = \underline{y}$. 
Then for each $p^i$, we have 
\[
y_{p^i} = x_{p^i}^p + px_{p^{i+1}} + pf_{p^i}(x_1,\ldots,x_{p^i}),
\]
where $f_{p^i}(x_1,\ldots, x_{p^i})$ is a polynomial with integer coefficients not involving the variable $x_{p^{i+1}}$.
\end{Lem}

\begin{Def}
For $p$ a prime, a ring $A$ is \emph{Witt-perfect at $p$} if for all positive integers $n$,
the map $F: W_{p^{n+1}}(A) \to W_{p^n}(A)$ is surjective. By \cite[Theorem~3.2]{DK12}, it is equivalent to require that the $p$-th power map on $A/pA$ is surjective and that for every $a \in A$, there exists $b \in A$ such that 
\[
b^p \equiv pa \bmod p^2.
\]
One important consequence 
is that the condition that $A$ is Witt-perfect at $p$ depends only on $A/p^2 A$; for instance, this implies
that $A$ is Witt-perfect at $p$ if and only if $A \otimes_{\ZZ} \ZZ_{(p)}$ is Witt-perfect at $p$.  See the aforementioned theorem for various other equivalent formulations
of the Witt-perfect condition. 
\end{Def}

The preceding comments make the following lemma clear.  We will use it repeatedly in what follows.

\begin{Lem} \label{sequence of p-power roots}
Let $A$ be a ring which is Witt-perfect at a prime $p$. Then there exist $x_1,x_2,\dots \in A$
such that $x_1^p \equiv p \pmod{p^2}$ and $x_{n+1}^p \equiv x_n \pmod{p}$ for $n \geq 1$.
\end{Lem}

\section{Almost purity}
\label{sec:almost purity}
We next set up the context for the general almost purity theorem stated in \cite{DK12},
giving a somewhat more detailed exposition than in \cite[Theorem~5.2]{DK12}.

\begin{Hyp}
Throughout \S\ref{sec:almost purity}, fix a prime number $p$
and let $A$ be a $p$-normal ring (see Definition~\ref{p-normal ring}).
\end{Hyp}

\begin{Def} \label{p-normal ring}
A ring $A$ is \emph{$p$-normal} if it is $p$-torsion-free and integrally closed in $A_p := A[p^{-1}]$.
By a \emph{$p$-ideal} in $A$, we will mean any ideal $I$ determining the $p$-adic topology on $A$:  this means $(p^m) \subseteq I$ and $I^n \subseteq (p)$ for some positive integers $m,n$.
\end{Def}

The following lemma should be thought of as providing explicit generators for the ideal of ``elements of $p$-adic valuation at least $\frac{m}{p^n}$'' for $\frac{m}{p^n} \leq 1$.

\begin{Lem} \label{power ideals}
Suppose that $A$ is Witt-perfect at $p$, and choose 
$x_1,x_2,\dots \in A$ as in Lemma~\ref{sequence of p-power roots}. Then
\[
\{x \in A: x^{p^n} \in p^m A \} = (x_n^m, p)A \qquad (n=1,2,\dots; m=1,2,\dots,p^n).
\]
\end{Lem}
\begin{proof}
We proceed by descending induction on $m$. For the base case $m=p^n$, we observe that
$\{x \in A: x^{p^n} \in p^{p^n} A \} = pA$ because $A$ is integrally closed in $A_p$ by hypothesis.
Given the claim for $m+1$, if $x \in A$ satisfies $x^{p^n} \in p^m A$, we may write
$x^{p^n} = p^m a$ and then choose $b \in A$ with $b^{p^n} \equiv - a \pmod{pA}$. We then have
$x^{p^n} \equiv - (x_n^m b)^{p^n} \pmod{p^{m+1} A}$. Now write
\[
(x + x_n^m b)^{p^n}  = x^{p^n} + (x_n^m b)^{p^n}  + \sum_{i=1}^{p^n-1} \binom{p^n}{i} x^i (x_n^m b)^{p^n-i}
\]
and observe that $(x^i (x_n^m b)^{p^n-i})^{p^n} \in p^{mp^n} A$, so by repeated application of the base case
we have $x^i (x_n^m b)^{p^n-i} \in p^m A$. It follows that $(x + x_n^m b)^{p^n} \in p^{m+1} A$,
so by the induction hypothesis $x + x_n^m b \in (x_n^{m+1}, p)A$. This completes the induction. 
\end{proof}
\begin{Cor}
With notation as in Lemma~\ref{power ideals},
the ideals $(x_n, p)$ in $A$ form a cofinal sequence among $p$-ideals.
Moreover, if $A$ is $p$-adically complete, the ideals $(x_n)$ and $(x_n,p)$ coincide.
\end{Cor}

\begin{Lem} \label{p-integral modulo nilradical}
Let $I$ be the nilradical of $A$. Then $I = IA_p$ and this is the nilradical of $A_p$.  Moreover,
the map $A_p/I \to (A/I)_p$ is bijective
and $A/I$ is $p$-normal.
\end{Lem}
\begin{proof}
Note first that every element of $IA_p$ is integral over $\ZZ$, and so belongs to $A$ because $A$
is $p$-normal. It is clear that $IA_p$ is the nilradical of $A_p$, $A/I$ is $p$-torsion-free,
and $A_p/I \to (A/I)_p$ is bijective. It remains to check that $A/I$ is integrally closed in $(A/I)_p \cong A_p/I$;
for this, choose $x \in A_p/I$ to be a root of the monic polynomial $P(T) \in (A/I)[T]$.
Lift $x$ to $y \in A_p$ and lift $P(T)$ to the monic polynomial $Q(T) \in A[T]$.
We then have $Q(y) \in I$, so $y$ is integral over $A$; therefore $y \in A$ and $x \in A/I$, as desired.
\end{proof}

\begin{Lem} \label{trace of integral extension}
Let $B$ be the integral closure of $A$ in a finite \'etale extension
of $A_p$. 
\begin{enumerate}
\item[(a)]
The map $\Trace_{B_p/A_p}$ carries $B$ into $A$.
\item[(b)]
For any finite set $b_1,\dots,b_n$ of elements of $B$ which generates $B_p$ as an $A_p$-module,
the quotient $B/(Ab_1 + \cdots + Ab_n)$ is killed by some power of $p$.
\end{enumerate}
\end{Lem}
\begin{proof}
To check (a), we may use the fact that every ring is the union of subrings which are finitely generated
over $\ZZ$ to reduce to the case where $A$ is noetherian.
We may then use Lemma~\ref{p-integral modulo nilradical} to replace $A$ with its reduced quotient.
Since now $A$ is noetherian and reduced, it has finitely many minimal prime ideals, none of which contain
$p$ because the latter is not a zero-divisor. We may thus replace $A$ with the product of the quotients
by the minimal prime ideals, then by an individual factor of the product. At this point, we need
only treat the case where $A$ is an integral domain; we may also assume that $B$ is an integral domain.
Let $K$ and $L$ be the respective fraction fields of $A$ and $B$.
Let $M$ be a Galois closure of $L$ over $K$.
For any $x \in B$, the conjugates of $x$ in $M$ satisfy every integral relation over $A$ satisfied by $x$
itself, and are thus integral over $A$. Consequently, any sum of Galois conjugates of $x$ is also integral over $A$, but this implies that $\Trace_{B_p/A_p}(x)$ is integral over $A$. Hence $\Trace_{B_p/A_p}(x)$ is in $A$, which proves the claim.

To check (b),
let $F$ be the free module over $A$ generated by $e_1,\dots,e_n$.
Let $F^* := \Hom_A(F,A)$ be the dual module and let $e_1^*, \dots, e_n^*$ be the corresponding dual basis of $F^*$. We then have the identity
\[
x = \sum_{i=1}^n e_i^*(x) e_i \qquad (x \in F_p).
\]
Let $\pi: F \to B$ be the $A$-module homomorphism taking $e_i$ to $b_i$,
and choose an $A_p$-linear splitting $\iota: B_p \to F_p$ of $\pi$. We then have
\[
b = \sum_{i=1}^n e_i^*(\iota(b)) b_i \qquad (b \in B).
\]
Since $B_p$ is finite \'etale over $A_p$, the trace pairing on $B_p$ is perfect; consequently,
there exists $t_i \in B_p$ such that
$e_i^*(\iota(b)) = \Trace_{B_p/A_p}(b t_i)$ for all $b \in B_p$. Choose a nonnegative integer $m$ for which
$p^m t_i \in B$ for $i=1,\dots,n$; then
\[
p^m b = \sum_{i=1}^n \Trace_{B_p/A_p}(b p^m t_i) b_i \qquad (b \in B).
\]
By (a), the quantities $\Trace_{B_p/A_p}(b p^m t_i)$ belong to $A$;
it follows that $p^m$ kills $B/(Ab_1 + \cdots + Ab_n)$.
\end{proof}

\begin{Lem} \label{almost splitting lemma}
Let $B$ be an $A$-module and fix an $A$-module endomorphism $t$ of $B$.
Suppose that for some $A$-module homomorphism
$\pi: F \to B$ with $F$ finite free, there exists an $A$-module homomorphism $\iota: B \to F$ such that 
$\pi \circ \iota = t$. Then for any other $A$-module homomorphism $\pi': F' \to B$ with $F'$ finite free
and with $\image(\pi) \subseteq \image(\pi')$, there exists an $A$-module homomorphism $\iota': B \to F'$
such that $\pi' \circ \iota' = t$.
\end{Lem}
\begin{proof}
Since $F$ is free and $\image(\pi) \subseteq \image(\pi')$, we may construct a homomorphism $F \to F'$ such that the composition $F \to F' \to B$
coincides with $\pi$. We then compose $B \to F \to F'$ to get the map $\iota'$.
\end{proof}

\begin{Lem} \label{descent for finite projective}
Let $B$ be the integral closure of $A$ in a finite \'etale extension
of $A_p$. Let $A',B'$ be the $p$-adic completions of $A,B$. Choose $a \in A$.
\begin{enumerate}
\item[(a)]
Suppose that there exist a finite free $A'$-module $F'$ and an $A'$-module homomorphism $F' \to B'$ with the following property: 
we can find
another $A'$-module homomorphism $B' \to F'$ such that the composition $B' \to F' \to B'$ is multiplication
by $a$. Then this can also be achieved with $F' \to B'$ being the base extension of an $A$-module
homomorphism $F \to B$, with $F$ a finite free $A$-module.  Moreover, it can be ensured that the cokernel of $F \to B$ is killed by a power of $p$.
\item[(b)]
Let $F \to B$ be an $A$-module homomorphism satisfying the conclusion of (a).
Then there exists an $A$-module homomorphism $B \to F$ such that the composition $B \to F \to B$ is multiplication by
$a$.
\end{enumerate}
\end{Lem}
\begin{proof}
To prove (a), note that by 
Lemma~\ref{trace of integral extension}, we can find a nonnegative integer $m$ and some elements
$b_1,\dots,b_k$ of $B$ such that $p^m B \subseteq b_1 A + \cdots + b_k A$;
we then also have $p^m B' \subseteq b_1 A' + \cdots + b_k A'$.
For this choice of $m$, write the images in $B'$ of the basis elements of $F'$ as
$t_1 + p^m u_1, \dots, t_h + p^m u_h$ with $t_i \in B$, $u_i \in B'$.
Let $G \to B$ be the $A$-module homomorphism with $G$ a finite free $A$-module whose basis elements map
to $b_1, \dots, b_k, t_1, \dots, t_h$, and put $G' := G \otimes_{A} A'$.  Notice that the image of $G \to B$ contains $p^m B$ and that the image of $G' \to B'$ contains the image of $F' \to B'$.  
We may then apply Lemma~\ref{almost splitting lemma} to conclude.

To prove (b), let $T$ be the set of $A_p$-module homomorphisms $B_p \to F_p$
such that the composition $B_p \to F_p \to B_p$ is multiplication by $a$.
This set is nonempty because $F_p \to B_p$ is surjective and $B_p$ is a finite projective $A_p$-module.
Choose an element $t_0 \in T$; the map $t \mapsto t - t_0$ then identifies $T$ with the
set of homomorphisms $B_p \to \ker(F_p \to B_p)$, which we may view as a finite projective $A_p$-module.
Let $T'$ be the set of $A'_p$-module homomorphisms $B'_p \to F'_p$
such that the composition $B'_p \to F'_p \to B'_p$ is multiplication by $a$;
the map $t \mapsto t - t_0$ then identifies $T'$ with $T \otimes_{A_p} A'_p$.
Since the conclusion of (a) has been assumed, there exists 
an element $t'$ of $T'$ which carries $B'$ into $F'$; 
by Lemma~\ref{trace of integral extension}, every element in a sufficiently small $p$-adic neighborhood of $t'$
also has this property. We can thus find $t \in T$ which carries $B'$ into $F'$ and thus carries $B$ into $F$.
This gives the desired maps $B \to F \to B$, proving (b).
\end{proof}

\begin{Thm}[Almost purity] \label{almost purity part 1}
Assume that $A$ is Witt-perfect at $p$.
Let $B$ be the integral closure of $A$ in a finite \'etale extension of $A_p$. 
\begin{enumerate}
\item[(a)]
The ring $B$ is again Witt-perfect at $p$.
\item[(b)]
For every $p$-ideal $I$, 
there exist a finite free $A$-module $F$ and an $A$-module homomorphism $F \to B$
whose cokernel is killed by $I$.
\item[(c)]
For any data as in (b), for each $t \in I$ there exists an $A$-module homomorphism
$B \to F$ such that the composition $B \to F \to B$ equals multiplication by $t$.
\item[(d)]
The trace pairing map $B \to \Hom_A(B,A)$ (which exists by
Lemma~\ref{trace of integral extension}) is injective and its cokernel is killed by every $p$-ideal of $A$.
\end{enumerate}
\end{Thm}
\begin{proof}
Part (a) is \cite[Theorem~5.2(a)]{DK12}. 
(We note two typos in that proof: In the last paragraph, the sentence
``Since $B$ satisfies (ii) and $\psi(B[p^{-1}])$ is dense in $S$...'' should read
``Since $\mathfrak{o}_B$ satisfies (ii) and $\psi(B)$ is dense in $S$....'')
We deduce the other parts 
%from the rest of \cite[Theorem~5.2]{DK12} 
as follows.

Choose $x_1,x_2,\dots \in A$ as in Lemma~\ref{sequence of p-power roots}.
Let $A_0$ be the subring of $A_p$ consisting of those $x$ for which $p x^n \in A$ for all $n \in \NN$.
As in the proof of \cite[Theorem~5.2(a)]{DK12}, $A_0$ is Witt-perfect at $p$.
Let $B_0$ be the integral closure of $A_0$ in $B_p$.
Let $A', B'$ be the $p$-adic completions of $A_0, B_0$; note that
$x_1,x_2,\dots$ are units in $A'_p$.
By \cite[Theorem~5.5.9]{KL11} or \cite[Theorem~5.25]{Sch11},
for each sufficiently large $n \in \NN$, there exist a finite free $A'$-module $F'$ and $A'$-module homomorphisms
$B'\to F' \to B'$ whose composition is multiplication by $x_{n+1}$.
By Lemma~\ref{descent for finite projective}(a), we may ensure that $F' \to B'$ is the base extension of a
homomorphism $F_0 \to B_0$ of $A$-modules; by Lemma~\ref{descent for finite projective}(b), we may ensure that the cokernel of $F_0 \to B_0$ is killed by $x_{n+1}$.

Since $x_{n+1}^{p-1} B_0 \subseteq B$, after multiplying by $x_{n+1}^{p-1}$ we may descend to a homomorphism
$F \to B$ of $A$-modules whose cokernel is killed by $x_{n+1}^p$.
Using Lemma~\ref{trace of integral extension}, we may add generators to force the cokernel to be killed also by $p^m$
for some nonnegative integer $m$.
Since the ideals $(x_{n+1}^p,p^m) = (x_n,p)$ are cofinal among $p$-ideals, this yields (b).

Continuing with notation as above, by Lemma~\ref{descent for finite projective}, there also exist a finite free $A_0$-module $F_0$ and $A_0$-module homomorphisms $B_0\to F_0 \to B_0$ whose composition is multiplication by $x_{n+1}$. If we multiply the map $B_0 \to F_0$ by $x_{n+1}^{p-1}$, the result carries $B$ into $F$; we thus obtain a finite free $A$-module $F$ and $A$-module homomorphisms $B \to F \to B$ whose composition is multiplication by
$x_{n+1}^p$. On the other hand, by Lemma~\ref{trace of integral extension}, we can make a different choice of
$B \to F$ such that the composition $B \to F \to B$ is multiplication by $p^m$ for some nonnegative integer $m$. We can thus arrange for the
composition to equal multiplication by $t$ for any $t \in (x_{n+1}^p, p^m) = (x_n,p)$.
By Lemma~\ref{descent for finite projective}(b) and the fact that the ideals $(x_n,p)$ are cofinal
among $p$-ideals, we deduce (c).

Note that $B \otimes_A B$ is Witt-perfect at $p$. To see this, note that by \cite[Corollary~3.3]{DK12}, it suffices to show that the $p$-power map on $B \otimes_A B$ is surjective modulo~$p$ and that $V(1)$ is contained in the image of $F: W_{p^2}(B \otimes_A B) \to W_p(B \otimes_A B)$.  These results hold because, by the same corollary, the analogous conditions hold for the ring $B$.  For example, to prove the second claim, consider the element $(b_1 \otimes 1, b_p \otimes 1, b_{p^2} \otimes 1)$, where $\underline{b} := (b_1, b_p, b_2) \in W_{p^2}(B)$ is such that $F(\underline{b}) = (1,0)$. Next note that $B_p$ is finite projective as a module over $B_p \otimes_{A_p} B_p \cong (B \otimes_A B)_p$ via the multiplication map. 

Let $C$ be the integral closure of $B \otimes_A B$ in $B_p \otimes_{A_p} B_p$.
By \cite[Proposition~3.6.11]{KL11}, $C/(B \otimes_A B)$ is killed by every $p$-ideal of $A$. We may show that $C$ is also Witt-perfect at $p$ by imitating an argument
from the proof of \cite[Theorem~5.2(a)]{DK12}; see Lemma~\ref{L:Witt-perfectness for integral closure} below.
We may thus apply (b) and (c) with $A$ replaced by $C$
to deduce (d).
\end{proof}
\begin{Cor}
Let $A$ be a $\ZZ$-torsion-free ring which is integrally closed in $A_{\QQ} := A \otimes_{\ZZ} \QQ$.
Let $B$ be the integral closure of $A$ in a finite \'etale extension of $A_{\QQ}$. If $A$ is Witt-perfect at $p$,
then so is $B$.
\end{Cor}
\begin{proof}
This follows from Theorem~\ref{almost purity part 1} and the fact that Witt-perfectness at $p$ can be checked
after inverting all primes other than $p$.
\end{proof}

\begin{Lem} \label{L:Witt-perfectness for integral closure}
Let $A$ be a $p$-torsion-free ring which is Witt-perfect at $p$. Let $B$ be the integral closure of $A$ in $A_p$. If $B/A$ is killed by every $p$-ideal of $A$, then $B$ is also Witt-perfect at $p$.
\end{Lem}
\begin{proof}
By \cite[Corollary~3.3]{DK12} again, there exist $x_1,x_2 \in A$ such that $x_1^p \equiv p \pmod{p^2}$ and $x_2^p \equiv x_1 \pmod{p}$,
and it suffices to show that the $p$-power map on $B$ is surjective modulo $(x_2^p,p)$.  (In general, if $B$ contains a ring which is Witt-perfect at $p$, then $B$ is Witt-perfect at $p$ if and only if the $p$-th power map on $B$ is surjective modulo~$p$.)
Note that $x_2^{p^2} \equiv p \pmod{p^2}$, so on one hand $x_2^{p^2}$ is divisible by $p$, and on the other hand there exists $u \in A$ such that
$x_2^{p^2} u \equiv p \pmod{p^p}$.

Choose any $x \in B$. Since $B/A$ is killed by the $p$-ideal $(x_2^{p^2-p}, p)$, we have $x_2^{p^2-p} x \in A$. Since $A$ is Witt-perfect at $p$, there exists $y \in A$ such that
$y^p \equiv x_2^{p^2-p} x \pmod{pA}$. Put $z = x_2^{p^2-p+1} u$ and $w = yz/p$.
On one hand, $z^p$ is divisible by $(x_2^{p^2})^{p-1} x_2^p$ and hence by $p^{p-1} x_2^p$;
on the other hand $x_2^{p-1} z = x_2^{p^2} u \equiv p \pmod{p^p A}$, so, raising both sides to the $p$-th power, 
$x_2^{p^2-p} z^p \equiv p^p \pmod{p^{p+1} A}$. Hence
\[
w^p -x =  \frac{y^p - x_2^{p^2-p} x}{p} \frac{z^p}{p^{p-1}} + \frac{x_2^{p^2-p} z^p - p^p}{p^{p+1}} (px)
\in (x_2^{p},p)A,
\]
so on one hand $w \in B$, and on the other hand $w^p \equiv x \pmod{(x_2^2 ,p)B}$.
\end{proof}

\begin{Rmk}
We do not know whether Lemma~\ref{L:Witt-perfectness for integral closure} remains true if we drop the condition that $B/A$ is killed by every $p$-ideal of $A$. A typical example for which this condition fails, but the conclusion remains true, is
\begin{align*}
A &= \mathcal{O}_{\CC_p}[p^{1/p^n} T^{1/p^n} : n = 0, 1, \ldots][ T^{2/p^n} : n = 0,1, \ldots] \\
B &= \mathcal{O}_{\CC_p}[T^{1/p^n} : n = 0, 1, \ldots].
\end{align*}
\end{Rmk}

\section{Inverse limits}

We next introduce a formal generalization of finite length Witt vectors by taking an inverse limit under the Witt vector Frobenius map.   Each of these new objects  possesses an infinite sequence of ghost components 
\[
\ldots, w_{p^{-2}},\, w_{p^{-1}},\, w_1.
\]
Here is the formal definition.

\begin{Def}
Let $\Warrow(R)$ denote the inverse limit of the inverse system consisting of one copy of the ring $W_{p^n}(R)$ for each $n \geq 0$, and one transition map
\[
F^{n_2 - n_1}: W_{p^{n_2}}(R) \rightarrow W_{p^{n_1}}(R)
\]
for each $n_2 \geq n_1 \geq 0$.  For any $n \geq 0$, we define the $(p^{-n})$-th ghost map 
\[
w_{p^{-n}}: \Warrow(R) \rightarrow R
\]
by projecting to the first Witt-component of $W_{p^n}(R)$.  By the definition of the Frobenius map in terms of ghost components, this is equivalent to defining $w_{p^{-n}}$ to be the composition of the projection of $\Warrow(R) \rightarrow W_{p^{n+i}}(R)$ followed by the $p^i$-th ghost map $w_{p^i}$ for any $i \geq 0$; the case $i = 0$ is exactly our definition. 
\end{Def}

\begin{Rmk}
There is a standard description of $p$-typical Witt vectors as an inverse limit of finite length Witt vectors.  The ring we construct is typically very different, because our transition maps are the Witt vector Frobenius maps, not restriction.  For example, using the Dwork lemma one can show that $\Warrow(\ZZ) \cong \ZZ$, while $W(\ZZ)$ is uncountable.
\end{Rmk}

\begin{Lem} \label{lifting lemma}
For any ring $R$ and any positive integer $m$,
the map induced by functoriality $\Warrow(R/(p^{m+1})) \to \Warrow(R/(p^m))$
is bijective.
\end{Lem}
\begin{proof}
Let $\underline{x}$ denote an arbitrary element in $\Warrow(R/(p^m))$, and let $\underline{x}_{p^{-i}}$ denote the corresponding element in $W_{p^i}(R/(p^m))$.  
For each $i$, choose an auxiliary lift 
\[
\underline{z}_{p^{-i}} \in W_{p^{i+1}}(R/(p^{m+1}))
\]
of $\underline{x}_{p^{-i}}$ by first lifting to $W_{p^i}(R/(p^{m+1}))$ and then lifting to $W_{p^{i+1}}(R/(p^{m+1}))$ along the natural restriction map (not along Frobenius). By Lemma~\ref{big Frobenius components lemma} and induction on $n$,
for $n=1,\dots,m+1$, the components of $F^n(\underline{z}_{p^{-i}})$ are
uniquely determined modulo $p^n$ by the element $\underline{x}$.
Thus the element
\[
\underline{y}_{p^{-i}} := F^{m+1}(\underline{z}_{p^{-(i+m+2)}}) \in W_{p^{i+1}}(R/(p^{m+1}))
\]
does not depend on the choices of the lifts $\underline{z}_{p^{-i}}$. To check that the sequence $(\underline{y}_{p^{-i}})$ is a coherent sequence under Frobenius, we must check that
\[
F(\underline{y}_{p^{-(i+1)}}) = \underline{y}_{p^{-i}}.
\]
In other words, we must check that 
\[
F\left(F^{m+1}(\underline{z}_{p^{-(i+m+3)}})\right) = F^{m+1}(\underline{z}_{p^{-(i+m+2)}}).
\]
Because $F(\underline{z}_{p^{-(i+m+3)}})$ is itself a lift of $\underline{x}_{p^{-(i+m+2)}}$, we may assume $F(\underline{z}_{p^{-(i+m+3)}}) = \underline{z}_{p^{-(i+m+2)}}$, and the claim follows.  

For any element $\underline{x} \in \Warrow(R/(p^m))$, we have constructed a lift in $\Warrow(R/(p^{m+1}))$.  We have also seen that the components of such a lift are uniquely determined, which completes the proof that the natural map $\Warrow(R/(p^{m+1})) \to \Warrow(R/(p^m))$ is a bijection.
\end{proof}

\begin{Cor} \label{inverse limit for complete}
Let $R$ be a ring which is $p$-adically complete.
Then the natural map $\Warrow(R) \to \Warrow(R/(p))$ is an isomorphism.
\end{Cor}

\begin{Lem} \label{inverse limit for perfect}
Suppose that the ring $R$ is of characteristic $p$.
Let $R'$ be the inverse limit of $R$ under the $p$-power
Frobenius endomorphism $\varphi$. 
\begin{enumerate}
\item[(a)]
The natural projection $\Warrow(R') \to \Warrow(R)$ is an isomorphism.
\item[(b)]
The map $F$ is bijective on $\Warrow(R')$, so there are natural isomorphisms
\[
W(R') \cong \varprojlim_{F} W(R') \cong \Warrow(R').
\]
\end{enumerate}
\end{Lem}
\begin{proof}
Both claims follow easily from the equality $F = W(\varphi)$.
\end{proof}

\begin{Thm} \label{theta map}
Let $p$ be a prime.
Let $R$ be a ring which is $p$-adically complete.
Let $R'$ be the inverse limit of $R/(p)$ under the $p$-power Frobenius endomorphism.
Form the composition 
\[
W(R') \cong \Warrow(R') \cong \Warrow(R/(p)) \cong \Warrow(R) \to W_{1}(R) = R
\]
in which the first isomorphism is the identification of Lemma~\ref{inverse limit for perfect}(b), the second isomorphism is the identification of Lemma~\ref{inverse limit for perfect}(a), the third map is the identification
of Corollary~\ref{inverse limit for complete}, and the fourth map is the projection onto the final term of the inverse system. The result is the usual \emph{theta map} $\theta: W(R') \to R$ of $p$-adic Hodge theory,
which intervenes in the construction of Fontaine's $p$-adic period rings; for example, see \cite[\S 4.4]{BC09}.   This ring homomorphism $\theta$ is surjective if 
$R$ is Witt-perfect at $p$.
\end{Thm}

\begin{proof}
The last statement, that $\theta$ is surjective if $R$ is Witt-perfect at $p$, is obvious:  By the definition of Witt-perfect, any element in $R = W_1(R)$ can be prolonged into an element of $\Warrow(R)$, hence the map $\Warrow(R) \rightarrow R$ is surjective, and hence so is the composite $W(R') \rightarrow R$.  

To prove the first statement, we need of course to recall the usual definition of $\theta: W(R') \rightarrow R$:
\[
\theta: \sum_{i = 0}^{\infty} p^i [x_{p^i}] \mapsto \sum_{i=0}^{\infty} p^i \lim_{n \rightarrow \infty} \left(\widetilde{x}_{p^i}^{(p^n)}\right)^{p^n},
\]
where $x_{p^i}^{(p^n)} \in R/p$ is the $p^n$-th component in the inverse system corresponding to $x_{p^i} \in R' = \varprojlim R/p$, and where $\widetilde{x}_{p^i}^{(p^n)} \in R$ is any lift.  The authors thank Ruochuan Liu for the following proof that this agrees with the description of $\theta$ from the statement of the theorem.

Let $\underline{y} \in \Warrow(R)$ denote the element corresponding to $\underline{x} \in W(R')$ under the above composition of isomorphisms.  Tracing through the definitions of these maps, we find that for our lift $\widetilde{x}_{p^i}^{(p^n)}$, we may take $y_{p^{-n-i}p^{i}}$.  Then we find
\begin{align*}
\sum_{i=0}^{\infty} p^i \lim_{n \rightarrow \infty} \left(\widetilde{x}_{p^i}^{(p^n)}\right)^{p^n}  &= \sum_{i=0}^{\infty} p^i \lim_{n \rightarrow \infty} \left(y_{p^{-n-i} p^{i}} \right)^{p^n} \\
&= \lim_{n \rightarrow \infty} \sum_{i = 0}^{n} p^i \left(y_{p^{-n-i} p^{i}} \right)^{p^n} \\
&= \lim_{n \rightarrow \infty} \sum_{i = 0}^{n} p^i \left( y_{p^{-n} p^i} \right)^{p^{n-i}}.
\end{align*}
This latter term is exactly the $p^n$-th ghost component of $\underline{y}_{p^{-n}}$.  Because the Witt vectors $y_{p^{-n}}$ form a compatible system under the Witt vector Frobenius, the $p^n$-th ghost component of $\underline{y}_{p^{-n}}$ is actually independent of $n$.  In particular, we may take $n = 0$, which corresponds to the projection $\underline{y} \mapsto y_1$, as required.
\end{proof}

\begin{Rmk} \label{only inverse Frobenius}
For general $R$, the ring $\Warrow(R)$ admits an inverse Frobenius map $F^{-1}$ taking $\underline{x} = (\underline{x}_{p^{-n}})_{n=0}^\infty$
to the inverse system with $\underline{x}_{p^{-n-1}}$ restricted to $W_{p^n}(R)$ for each $n \geq 1$.  In general, $\Warrow(R)$ does not admit a Frobenius map.  For example, let $R = \ZZ_p[\mu_{p^{\infty}}]$ and let $R'$ denote its $p$-adic completion.  Let $r \in R' \setminus R$.   Define $\underline{x}_1 \in W_1(R')$ to be the length one Witt vector $(r)$.  Now inductively assume we have a Witt vector $\underline{x}_{p^{-(n-1)}} \in W_{p^{n-1}}(R')$ with ghost components $(w_1, w_p, \ldots, w_{p^{n-2}}, r)$ and $w_{p^i} \in R$ for $0 \leq i \leq n-2$, i.e., all ghost components except the last are in $R$.  Because $R'$ is Witt-perfect at $p$ (see \cite[Example~4.9]{DK12}), we can produce a Witt vector $\underline{x}_{p^{-n}} \in W_{p^n}(R')$ such that $F(\underline{x}_{p^{-n}}) = \underline{x}_{p^{-(n-1)}}$.  By \cite[Corollary~2.6]{DK12}, we can further assume that the first ghost component of $\underline{x}_{p^{-n}}$ is in $R$.  Induction then yields an element $\underline{x} \in \Warrow(R')$, and by construction $\underline{y} := F^{-1}(\underline{x}) \in \Warrow(R)$. However  $F(\underline{y}) \not\in \Warrow(R)$: the only possibility for its $w_1$ ghost component is $r \not\in R$; there are no other possibilities because $R'$ is $p$-adically separated. 
\end{Rmk}

\section{Norms on rings}

Before continuing, we need to set some notation and terminology concerning norms on rings.

\begin{Def}
For $A \in \Ab$, a \emph{nonarchimedean seminorm} on $A$ is a function $\left| \bullet \right|: A \to [0, +\infty)$
such that $\left| 0 \right| = 0$ and satisfying the strong triangle inequality $\left|x + y\right| \leq \max\{\left|x\right|,\left|y\right|\}$ for all $x,y \in A$.  
We say two seminorms $\left| \bullet \right|_1, \left| \bullet \right|_2$ on $R$ are \emph{metrically equivalent} (or just \emph{equivalent} when this is unambiguous)
if there exist constants $c_1, c_2 > 0$ such that for all $x \in A$,
$\left| x \right|_1 \leq c_1 \left| x \right|_2$ and
$\left| x \right|_2 \leq c_2 \left| x \right|_1$.
A seminorm $\left| \bullet \right|$ is a \emph{norm} if the inverse image of $0$ is equal to $\{0\}$.
For example, any seminorm which is metrically equivalent to a norm is itself a norm.
\end{Def}

\begin{Def}
For $R \in \Ring$, a seminorm $\left| \bullet \right|$ on $R$ is \emph{submultiplicative} if for
all $x,y \in R$, $\left| xy \right| \leq \left| x \right| \left| y \right|$. Note that if there exists a constant
$c > 0$ such that $\left| xy \right| \leq c \left| x \right| \left| y \right|$, then $\left| \bullet \right|$
is metrically equivalent to the submultiplicative seminorm $c^{-1} \left| \bullet \right|$.

The seminorm $\left| \bullet \right|$ is \emph{power-multiplicative} if it is submultiplicative and
additionally for all $x \in R$, $\left| x^2 \right| \geq \left| x \right|^2$. This implies that
$\left| x^n \right| = \left| x \right|^n$ for all $x \in R$, $n \in \NN$; in particular, if $R$ is nonzero and the norm $\left| \bullet \right|$ is nonzero, then
$\left| 1 \right| = 1$. We say $\left| \bullet \right|$ is \emph{uniform} if it is metrically
equivalent to a power-multiplicative seminorm.

%For any submultiplicative seminorm $\left| \bullet \right|$, the \emph{spectral seminorm} is the power-multiplicative seminorm $\left| \bullet \right|_{\mathrm{sp}} = \lim_{n \to \infty} \left| \bullet^n \right|^{1/n}$. Beware that this may be only a seminorm if we start with a nonuniform norm; for example, only a reduced ring can carry a power-multiplicative norm.

The seminorm $\left| \bullet \right|$ is \emph{multiplicative} if for all $x,y \in R$,
$\left| x y \right| = \left| x \right| \left| y \right|$.
\end{Def}

\begin{Def} \label{p-adic norm on a ring}
Let $A$ be a $p$-normal ring.
We define the \emph{$p$-adic seminorm} on $R := A_p$ as the 
power-multiplicative nonarchimedean seminorm $\left| \bullet \right|$ on $R$ computed as follows:
for $x \in R$, $|x|$ is the infimum of
$p^{-c/d}$ over all $c,d \in \ZZ$ with $d>0$ and $p^{-c} x^d \in A$.
Note that $A$ is contained in the subring of $R$ of elements of $p$-adic norm at most $1$, but these need not coincide.
On the other hand, every element of $R$ of norm strictly less than $1$ does belong to $A$.
\end{Def}

\begin{Def}
By an \emph{analytic field}, we will mean a field $K$
equipped with a multiplicative norm under which it is complete.
We say $K$ is \emph{nonarchimedean} if its norm is nonarchimedean. 
\end{Def}

We will also need Berkovich's generalization of the Gel'fand representation theorem in the proof of Proposition~\ref{power-multiplicative norm on W} below.
\begin{Thm} \label{Gelfand representation}
Let $R$ be a ring equipped with a power-multiplicative nonarchimedean norm. 
Then there exists an isometric embedding
of $R$ into a product of fields such that each field is equipped with a multiplicative nonarchimedean norm and the product
is equipped with the supremum norm.  
\end{Thm}
\begin{proof}
We may assume that $R$ is complete, in which case see \cite[Theorem~1.3.1]{Ber90}.
\end{proof}

\section{Norms on Witt vectors}
\label{sec:norms}

\begin{Hyp}
Throughout \S\ref{sec:norms}, 
suppose that $R \in \Ring$ is equipped with a nonarchimedean submultiplicative seminorm $\left|\bullet \right|$.
\end{Hyp}

\begin{Def}
We define the function $\left|\bullet \right|_{W}$ on $W(R)$ by the formula
\[
\left|(x_{p^n})_{n \geq 0}\right|_{W} = \sup\left\{ \left|x_{p^n}\right|^{1/p^n}: n \geq 0\right\}.
\]
Similarly for $(x_{p^n})_{0 \leq n \leq r} \in W_{p^r}(R)$, we put 
\[
\left|(x_{p^n})_{0 \leq n \leq r} \right|_W = \sup \left\{ \left| x_{p^n} \right|^{1/p^n} : 0 \leq n \leq r\right\}.
\]
\end{Def}

\begin{Rmk} \label{trivial bounds for W}
We will make frequent use of two trivial observations. On one hand, for any $n \geq 0$,
any ring $R$, and any $\underline{x} \in W(R)$ projecting to $\underline{y} \in W_{p^n}(R)$, we have
$\left|\underline{x}\right|_{W} \geq \left|\underline{y}\right|_{W}$. On the other hand, any $\underline{y} \in W_{p^n}(R)$
lifts to some $\underline{x} \in W(R)$ with $\left|\underline{x}\right|_{W} = \left|\underline{y}\right|_{W}$,
e.g., by padding with zeroes.
\end{Rmk}

\begin{Prop} \label{Frobenius Verschiebung bound infinite case} Let $\underline{x}, \underline{y} \in W(R)$.  We have the following:
\begin{enumerate}
\item[(a)] $\left|\underline{x} + \underline{y}\right|_{W} \leq \max\{ \left|\underline{x}\right|_{W}, \left|\underline{y}\right|_{W}\}$;
\item[(b)] $\left|\underline{x}  \underline{y}\right|_{W} \leq \left|\underline{x}\right|_{W} \left|\underline{y}\right|_{W}$;
\item[(c)] $\left|F(\underline{x})\right|_W \leq \left| \underline{x}\right|_{W}^p$;
\item[(d)] $\left| V(\underline{x})\right|_W = \left| \underline{x} \right|^{1/p}$.
\end{enumerate}
\end{Prop}

\begin{proof}
We first prove (a).  Let $M = \max\{ \left|\underline{x}\right|_{W}, \left|\underline{y}\right|_{W}\}$.  We are done if we show that the $p^i$-th component of $\underline{x} + \underline{y}$ has norm at most $M^{p^i}$.  We know the same is true for the components of $\underline{x}$ and $\underline{y}$ individually, and so we are done by Corollary \ref{homogeneous corollary}, together with the fact that the seminorm on $R$ is submultiplicative.   (Because the coefficients are in $\ZZ$, and our seminorm is nonarchimedean, the coefficients can be ignored for this calculation.)

The proof of (b) is similar. By Corollary \ref{homogeneous corollary} again, the $p^i$-th component of $\underline{x} \underline{y}$ is a product of a homogeneous degree $p^i$-polynomial in the $x$-variables with a homogeneous degree $p^i$-polynomial in the $y$-variables (with our usual weighting).  We are now finished because the seminorm on $R$ is nonarchimedean and submultiplicative.  

One proves (c) similarly.  The proof of (d), on the other hand,
is apparent from the direct expression of Verschiebung in terms of Witt components given in Theorem~\ref{construct Witt vectors}(c).
\end{proof}

Using Remark \ref{trivial bounds for W}, this proposition easily adapts to finite length Witt vectors. 

\begin{Cor} \label{Frobenius Verschiebung bound} Let $\underline{x}, \underline{y} \in W_{p^n}(R)$.  We have the following:
\begin{enumerate}
\item[(a)] $\left|\underline{x} + \underline{y}\right|_{W} \leq \max\{ \left|\underline{x}\right|_{W}, \left|\underline{y}\right|_{W}\}$;
\item[(b)] $\left|\underline{x}  \underline{y}\right|_{W} \leq \left|\underline{x}\right|_{W} \left|\underline{y}\right|_{W}$;
\item[(c)] $\left| F(\underline{x}) \right|_W \leq \left| \underline{x} \right|_W^p$;
\item[(d)] $\left| V(\underline{x})\right|_W = \left| \underline{x} \right|_W^{1/p}$.
\end{enumerate}
\end{Cor}

\begin{proof}
These trivially follow from Proposition~\ref{Frobenius Verschiebung bound infinite case} by using Remark \ref{trivial bounds for W}.  For example, to prove (b), let $\underline{x'}, \underline{y'}$ denote lifts to infinite length Witt vectors having the same norms as $\underline{x}$, $\underline{y}$.  Then we have
\[
\left|\underline{x} \underline{y}\right|_W \leq \left|\underline{x'} \underline{y'}\right|_W \leq \left|\underline{x'} \right|_W \left|\underline{y'}\right|_W = \left|\underline{x} \right|_W \left|\underline{y}\right|_W.
\]
\end{proof}

\begin{Prop} \label{power-multiplicative norm on W}
Assume the nonarchimedean submultiplicative seminorm on $R$ is a power-multiplicative norm.  Then the submultiplicative seminorm $\left| \bullet \right|_W$ on the subring of $\underline{x} \in W(R)$ for which $\left| \underline{x} \right|_W < +\infty$
is power-multiplicative.
\end{Prop}

\begin{proof}
By Theorem \ref{Gelfand representation}, there exists an isometric embedding of $R$ into a product of fields equipped with \emph{multiplicative} norms, and such that the norm on the product is the supremum of the norms on the factors.  It is clear that the Witt vector functor commutes with (possibly infinite) products of rings, so it suffices to prove the result in the case $R = K$, a nonarchimedean analytic field.  We further assume $K$ is algebraically closed (and hence so too is its residue field $k$).  

Let $\underline{x} \in W(K)$; we want to show that $\left| \underline{x} \right|_W^2 \leq \left| \underline{x}^2 \right|_W$.  We may reduce to the case where $\underline{x}$ has only finitely many nonzero components
(by successive approximations).  From the formula 
\[
[r] \cdot (x_1, x_p, \ldots) = (rx_1, r^p x_p, \ldots),
\]
and from the fact that the norm on $K$ is multiplicative, we see that 
\[
\left|[r] \cdot \underline{x} \right|_W = \left| r\right| \left| \underline{x} \right|_W. 
\]
In this way (using that $\underline{x}$ has only finitely many nonzero components), we can rescale to assume that $|\underline{x}|_W = 1$.  Thus we are reduced to proving that if $\underline{x} \in W(K)$ has $\left|\underline{x}\right|_W = 1$, then $\left| \underline{x}^2 \right|_W = 1$.  We know $\left| \underline{x}^2 \right|_W \leq 1$ by submultiplicativity.  If the algebraically closed residue field $k$ has characteristic different from $p$, then the ring $W(k)$ is isomorphic to a product of copies  of $k$, and in particular the projection of $\underline{x}^2$ is nonzero.  If $k$ has characteristic $p$, then from the classical theory of Witt vectors over perfect characteristic $p$ fields, we know $W(k)$ is reduced, and hence the projection of $\underline{x}^2$ is again nonzero.  This implies $\left| \underline{x}^2 \right|_W = 1$, as required.
\end{proof}

The Witt-perfect condition translates into a condition on norms of Witt vectors as follows.  Note that this result concerns $p$-typical Witt vectors evaluated on rings in which $p$ is a unit.  Typically such Witt vectors are not useful, but we are able to make meaningful statements by placing restrictions on their norms.

\begin{Lem} \label{Witt-perfect to norm-perfect}
Let $A$ be a ring which is $p$-normal
and Witt-perfect at $p$. Equip $R := A_p$ with the $p$-adic seminorm. 
Then for every $\underline{x} \in W_{p^n}(R)$, there exists $\underline{y} \in W_{p^{n+1}}(R)$ with $F(\underline{y}) = \underline{x}$ and $\left| \underline{y} \right|_{W}^p \leq \left| \underline{x} \right|_{W}.$
\end{Lem}

\begin{proof}
For notational convenience, we will replace $n$ by $s$ in the proof. So let us choose an $s \geq 0$, an element $\underline{x} \in W_{p^s}(R)$,
and put $c := \left| \underline{x} \right|_{W}$.
We wish to exhibit $\underline{y} \in W_{p^{s+1}}(R)$ with $F(\underline{y}) = \underline{x}$
and $\left| \underline{y} \right|_{W}^{p} \leq c$.
Note that for any $n \in \ZZ$, $F([p^n] \underline{y}) = [p^{p n}] \underline{x}$; consequently,
there is no loss of generality in assuming that $c > 1$.

Choose $x_1, x_2,\dots \in A$ as in Lemma~\ref{sequence of p-power roots}.
Then for any $n\geq 1$ and any $y \in R$, $\left| x_n y \right| = p^{-1/p^n} \left| y \right|$.
Choose integers $m,n \geq 0$ such that $ p^{-1/p^{s+1}} < p^{- m/p^{n-1}} c < 1$,
and put $\underline{x}' := [x_{n}^{p m}] \underline{x}$. Then
$\left| \underline{x}' \right|_{W} = p^{- m/p^{n-1}} c < 1$, so because $A$ is Witt-perfect,
we may lift
$\underline{x}'$ along $F$ to some $\underline{y}' \in W_{p^{s+1}}(A)$.  By  Lemma~\ref{big Frobenius components lemma}, for all $0 \leq j \leq s$, 
we have $|x'_{p^j} - (y'_{p^j})^p| \leq p^{-1}$,
and so $\left| \underline{y}' \right|_{W}^p \leq p^{-m/p^{n-1}} c$.

Because $x_n$ is a unit in $A_p$ whenever $A$ is $p$-adically complete,
we can choose $\underline{y}_0 \in W_{p^{s+1}}(R)$
so that $\left| [x_{n}^m]\underline{y}_0 - \underline{y}' \right|_{W} < 
p^{-m/p^{n}}$.
Considering the two cases $\left| [x_{n}^m] \underline{y}_0 \right|_{W} > \left| \underline{y}' \right|_{W}$ and $\left| [x_{n}^m] \underline{y}_0 \right|_{W} \leq \left| \underline{y}' \right|_{W}$, one checks that  
\begin{align*}
\left| [x_{n}^m] \underline{y}_0 \right|_{W} &\leq p^{-m/p^{n}} c^{1/p}, \\
\intertext{which means}
\left| \underline{y}_0\right|_W^p &\leq c. 
\end{align*}
Applying $F$ to  $[x_{n}^m]\underline{y}_0 - \underline{y}'$, we have by 
Corollary~\ref{Frobenius Verschiebung bound} that
\[
\left| [x_{n}^{p m}] (F(\underline{y}_0) - \underline{x}) \right|_{W} < p^{-m/p^{n-1}},
\]
and so $\left| F(\underline{y}_0) - \underline{x} \right|_{W} < 1$.
By applying the fact that $A$ is Witt-perfect again,
we can find $\underline{y}_1 \in W_{p^{s+1}}(R)$ with
$F(\underline{y}_1) = F(\underline{y}_0) - \underline{x}$ and 
$\left| \underline{y}_1 \right|^p_{W} \leq 1$.  We may then take $\underline{y} := \underline{y}_0 - \underline{y}_1$ to complete the proof.
\end{proof}

\section{Overconvergent Witt vectors}
\label{sec:overconvergent}

For rings of $p$-typical Witt vectors over rings of characteristic $p$, the importance of
subrings defined by growth conditions (often called \emph{overconvergent subrings}) has already
been observed in $p$-adic cohomology (e.g., see \cite{DLZ11a, DLZ11b}) and in $p$-adic Hodge theory
(e.g., see \cite{KL11}). One can introduce similar rings more generally,
but the most natural context for this is the inverse limit construction we introduced earlier.

\begin{Def} \label{norm on p-typical inverse limit}
Let $R$ be a ring equipped with a submultiplicative nonarchimedean seminorm $\left| \bullet \right|$.
For $b>0$, let $\Warrow^b(R)$ be the set of elements $(\underline{x}_{p^{-n}}) \in \Warrow(R)$ for which $p^{-bn} \left|\underline{x}_{p^{-n}} \right|^{p^n}_{W} \to 0$ as $n \to \infty$.
This forms an abelian group on which
the function $\left| \bullet \right|_{W,b}$ defined by the formula
\begin{equation} \label{equation for inverse limit norm}
\left| \underline{x} \right|_{W,b} = \sup_n \left\{ p^{-bn} \left|\underline{x}_{p^{-n}} \right|^{p^n}_{W} \right\}
\end{equation}
is a seminorm thanks to Corollary~\ref{Frobenius Verschiebung bound}.
Let $\Warrow^\dagger(R)$ denote the union of the $\Warrow^b(R)$ over all $b>0$; this forms a subring of $\Warrow(R)$ again by
Corollary~\ref{Frobenius Verschiebung bound}.
\end{Def}

Our first result shows that our definition is a generalization of the overconvergence condition studied in \cite{DLZ11a, DLZ11b}.

\begin{Thm} \label{overconvergence in characteristic p}
Keep notation as in Definition~\ref{norm on p-typical inverse limit}, 
and suppose further that $R$ is perfect of characteristic $p$
and that the seminorm on $R$ is power-multiplicative.  In this case, the seminorm $\left| \bullet \right|_{W,b}$ is power-multiplicative, and
via the natural isomorphism
$W(R) \cong \Warrow(R)$ of Lemma~\ref{inverse limit for perfect}(b),
$\Warrow^b(R)$ corresponds to the subring of $W(R)$ consisting of those $\underline{x}$
for which $p^{-bn} |x_{p^n}|^{p^{-n}}\to 0$ as $n \to \infty$.
We thus recover the usual \emph{overconvergent Witt ring} in characteristic $p$, as used for instance in
\cite{DLZ11a, DLZ11b}.   
\end{Thm}

\begin{proof}
Let $\underline{y} \in \Warrow(R)$ denote the image of some element $\underline{x} \in W(R)$ under the isomorphism $W(R) \cong \Warrow(R)$.  Then 
\begin{align*}
\underline{y}_{p^{-n}} &= (x_1^{p^{-n}}, x_p^{p^{-n}}, \ldots, x_{p^n}^{p^{-n}}). \\
\intertext{Then using the definitions of the various seminorms, we compute}
\left| \underline{y} \right|_{W,b} &=  \sup_n \left\{ p^{-bn} \left|\underline{y}_{p^{-n}} \right|^{p^n}_{W} \right\} \\
&= \sup_{j \leq n} \left\{ p^{-bn} \left|x_{p^{j}}^{p^{-n}} \right|^{p^{n-j}} \right\} \\
&= \sup_{j \leq n} \left\{ p^{-bn} \left|x_{p^{j}} \right|^{p^{-j}} \right\} \\
&= \sup_{j} \left\{ p^{-bj} \left|x_{p^{j}} \right|^{p^{-j}} \right\}. 
\end{align*}
This proves the stated description of $\Warrow^b(R)$.  We deduce that $\left| \bullet \right|_{W,b}$ is power-multiplicative as in the proof of \cite[Lemma~4.1]{Ked13a}.  

We now prove that this notion of overconvergence agrees with the notion used in \cite{DLZ11a, DLZ11b}.  Let $R = k\left[x_1^{\frac{1}{p^{\infty}}}, \ldots, x_r^{\frac{1}{p^{\infty}}} \right]$ denote the perfection of $k[x_1, \ldots, x_r]$, where $k$ is some perfect field of characteristic $p$.  A Witt vector $(f_1, f_p, f_{p^2},\ldots)\in W(R)$ is considered overconvergent if there exist positive constants $C,D$ such that 
\[
\deg f_{p^j} \leq Cjp^j + Dp^j
\]
for all $j \geq 0$.  If we define a norm on $R$ by setting 
\[
| f | := p^{\deg f},
\]
then it is clear that the two notions of overconvergence do in fact coincide.
\end{proof}

In the remainder of this section, we use the characteristic $p$ case to derive analogous properties in case $R$ carries a $p$-adic seminorm.

\begin{Hyp}
For the remainder of \S\ref{sec:overconvergent}, take $R$ to be $A_p$  equipped with the $p$-adic seminorm for some $p$-normal ring $A$.
\end{Hyp}

\begin{Lem} \label{effect of multiplication by p}
Let $\underline{x}  = (\underline{x}_{p^{-n}}) \in \Warrow(R)$ be an element such that 
$\underline{x}_{p^{-n}} = [y_{p^{-n}}]$ for some $y_{p^{-n}} \in R$. 
Then for each nonnegative integer $m$,
\begin{equation} \label{effect of multiplication by p}
\left| p^m \underline{x} \right|_{W,b} = p^{-\min\{b,1\}m} \left| \underline{x}_1 \right|_W.
\end{equation}
\end{Lem}
\begin{proof}
We may use homogeneity 
(Lemma~\ref{homogeneous polys})
to reduce to the case $\underline{x} = 1$. In $W(\ZZ)$, we have
\[
p^m = (z_0, z_1, \dots)
\]
for some $z_i$ with $z_0 = p^m$, $z_i \in p^{m-i} \ZZ$ for $i=1,\dots,m-1$, and $z_m \equiv 1 \pmod{p}$. Write
\[
\left| p^m \right|_{W,b} = \sup_n \{p^{-bn} \left| p^m \right|_W^{p^n} \} = \sup_{0 \leq i \leq n} \{ p^{-bn} \left| z_i \right|^{p^{n-i}}\}.
\]
In the final supremum, for $b \geq 1$, the term for a given $i$ and $n$
is at most 
\[
p^{-bn-(m-i)p^{n-i}} \leq p^{-n-(m-i)p^{n-i}} \leq p^{-i-(m-i)} = p^{-m}
\]
with equality for $i=n=0$; for $b < 1$, the term is at most
\[
p^{-bn-(m-i)p^{n-i}} \leq p^{-bn-b(m-i)p^{n-i}} \leq p^{-bi-b(m-i)} = p^{-bm}
\]
with equality for $i=m=n$. 
This proves the claim.
\end{proof}

\begin{Lem} \label{complete Witt vectors}
If $A$ is complete with respect to the seminorm $\left| \bullet \right|$, then $W_{p^n}(A)$ is complete with respect to the seminorm $\left| \bullet \right|_W$.  Similarly, $\Warrow^b(R)$ is complete with respect to the seminorm $\left| \bullet \right|_{W,b}$.  
\end{Lem}

\begin{proof}
Let $\underline{x}^{(1)}, \underline{x}^{(2)}, \ldots$ denote a Cauchy sequence of Witt vectors.  Using induction on $j$, one can check that the sequence of Witt vector components $x^{(1)}_{p^j}, x^{(2)}_{p^j}, \ldots$ is also a Cauchy sequence, hence it has a limit in $A$, and this is the limit of our original sequence of Witt vectors.  The proof for $\Warrow^b(R)$ is similar; the only additional detail to check is that the limits remain compatible under Frobenius; this follows from the fact that the Frobenius map can be defined in terms of polynomials of the coefficients, and hence $F: W_{p^{n+1}}(R) \rightarrow W_{p^n}(R)$ is continuous.
\end{proof}

\begin{Def} \label{transfer to positive characteristic}
Assume $R$ is $p$-adically complete and let $R'$ be the inverse limit of $R$ under the $p$-power map.
This set admits a natural ring structure as follows.
For $\underline{x}  = (x_{p^{-m}})_{m=0}^\infty$, 
$\underline{y} = (y_{p^{-m}})_{m=0}^\infty$ in $R'$, the sum $\underline{x} + \underline{y}$ is defined as the coherent sequence
$(z_{p^{-m}})_{m=0}^\infty$ with
\[
z_{p^{-m}} = \lim_{l \to \infty} (x_{p^{-m-l}} + y_{p^{-m-l}})^{p^l};
\]
the limit exists because if $\left| x_{1} \right|, \left| y_{1} \right| \leq c$, then
\[
\left| (x_{p^{-m-l-1}} + y_{p^{-m-l-1}})^{p^{l+1}} - (x_{p^{-m-l}} + y_{p^{-m-l}})^{p^l} \right| \leq p^{-l} c.
\]
For this addition structure, the ring $R'$ has characteristic~$p$.
By a similar calculation, we have $z_{p^{-m-1}}^p = z_{p^{-m}}$, so we get a well-defined element of $R'$. We thus obtain the ring structure on $R'$; by similar calculations, the formula
\[
\left| \underline{x} \right| := \left| x_{1} \right|
\]
defines a power-multiplicative seminorm on $R'$.
\end{Def}

\begin{Thm} \label{compare to positive characteristic1}
Suppose that $A$ is Witt-perfect at $p$ and $p$-adically complete.
Let $A'$ be the inverse limit of $A/(p)$ under the $p$-power Frobenius endomorphism. For any $b$, there is a natural way to embed $\Warrow(A')$ into $\Warrow^b(R')$, and for $0 < b \leq 1$,
the isomorphism $\Warrow(A') \cong \Warrow(A)$ from
Theorem~\ref{theta map} extends to an isometric isomorphism $\Warrow^b(R') \cong \Warrow^b(R)$.
\end{Thm}
\begin{proof}
There is an obvious map to $A'$ from the subring of $\underline{x} \in R'$ with $x_1 \in A$; one checks that this map is an isomorphism by following the proof of Lemma~\ref{lifting lemma}.

We next construct the map $\Warrow^b(R') \to \Warrow(R)$. By Theorem~\ref{overconvergence in characteristic p}, we may identify $\Warrow^b(R')$ with the ring of  $\underline{x} \in W(R')$ for which $p^{-bn} \left| x_{p^n} \right|^{p^{-n}} \to 0$ as $n \to \infty$. Each $x_{p^n}$ corresponds to a coherent sequence $(x_{p^n,p^{-m}})_{m=0}^\infty$ of elements of $R$
in which $x_{p^n,p^{-m-1}}^p = x_{p^n,p^{-m}}$ for all $m$. 
We may then define $\underline{y}_{p^n}  \in \Warrow(R)$ as the coherent sequence $([x_{p^n,p^{-m-n}}])_{m=0}^\infty$ and then define the map in question so as to send $\underline{x}$ to $\underline{y} := \sum_{n=0}^\infty p^n \underline{y}_{p^n}$. Note that \eqref{effect of multiplication by p} and Lemma~\ref{complete Witt vectors} together
guarantee both that the sum converges and that the resulting map is submetric. To check that the map is isometric, pick the smallest index $h \geq 0$ which maximizes $p^{-bh} \left| x_{p^h} \right|^{p^{-h}}$
(assuming that $\underline{x}$ is nonzero, as otherwise there is nothing to check),
and put $\underline{z} := \sum_{n=h}^\infty p^{n} \underline{y}_{p^n}$.
Then on one hand, by the previous discussion we have $\left| \underline{y} - \underline{z}\right|_{W,b} < \left| \underline{x} \right|_{W,b}$. On the other hand, 
for any $m \geq 0$ we have
$z_{p^{-m},1} = \sum_{n=h}^\infty p^{n} x_{p^n,p^{-m-n}}$.
By our choice of $h$, for all $n>h$ and $m>0$ we have
\begin{align*}
p^{-bh} \left| x_{p^h,p^{-m-h}} \right| &= p^{bh(p^{-m}-1)} \left( p^{-bh} \left| x_{p^h} \right|^{p^{-h}} \right)^{p^{-m}} \\
&\geq p^{bh(p^{-m}-1)} \left( p^{-bn} \left| x_{p^n} \right|^{p^{-n}} \right)^{p^{-m}}\\
&\geq p^{b(1-p^{-m})} p^{bn(p^{-m}-1)} \left( p^{-bn} \left| x_{p^n} \right|^{p^{-n}} \right)^{p^{-m}}\\
& = p^{b(1-p^{-m})} p^{-bn} \left| x_{p^n,p^{-m-n}} \right|,
\end{align*}
so $z_{p^{-m},1}$ is dominated by the summand with $n=h$.
Hence $\left| \underline{z} \right|_{W,b} \geq \left| \underline{x}
\right|_{W,b}$, so $\left| \underline{y} \right|_{W,b} \geq \left| \underline{x} \right|_{W,b}$ as well. Consequently, the defined map is an isometry, and in particular is injective.

It is straightforward to check that this map extends the isomorphism $\Warrow(A') \cong \Warrow(A)$ by tracing through the construction of Theorem~\ref{theta map}. In particular, if we choose a sequence $x_1,\dots \in A$ according to Lemma~\ref{sequence of p-power roots}, then these define an element of $A'$. Let $\underline{y} \in W(A') \cong \Warrow(A')$ be the Teichm\"uller lift of this element; it is a unit in $\Warrow^b(R')$. The image of the subring 
$\Warrow(A')[\underline{y}^{-1}]$ of $\Warrow^b(R')$
equals the subring of $\underline{x} \in \Warrow^b(R)$ 
for which $\sup_n \{\left| \underline{x}_{p^{-n}} \right|^{p^n}_W\} < +\infty$. By Lemma~\ref{Witt-perfect to norm-perfect}, this subring is dense in $\Warrow^b(R)$: given $\underline{x} \in \Warrow^b(R)$, we may use the lemma to lift $\underline{x}_{p^{-n}}$ to the subring, and these lifts converge to $\underline{x}$. It follows that the map
$\Warrow^b(R') \to \Warrow^b(R)$ is surjective.
\end{proof}

To study $\Warrow^b(R)$ for $b>1$, we use the inverse Frobenius map
described in Remark~\ref{only inverse Frobenius}.
\begin{Lem} \label{effect of inverse Frobenius}
For $b \geq 1$ and $\underline{x} \in \Warrow^b(R)$,
we have $F^{-1}(\underline{x}) \in \Warrow^{b/p}(R)$ and
\begin{equation} \label{equation for inverse Frobenius norm}
\max\{\left| \underline{x}_1 \right|_W, p^{-b} \left| F^{-1}(\underline{x}) \right|_{W,b/p}^p\} \leq
\left| \underline{x} \right|_{W,b} \leq \max\{\left| \underline{x}_1 \right|_W,\left| F^{-1}(\underline{x}) \right|_{W,b/p}^p\}.
\end{equation}
\end{Lem}
\begin{proof}
Put $\underline{y} = F^{-1}(\underline{x})$; then 
for $n \geq 0$, $\underline{y}_{p^{-n}}$ is the restriction of
$\underline{x}_{p^{-n-1}}$ from $W_{p^{n+1}}(R)$ to $W_{p^n}(R)$.
From this, it is clear that
\[
\left| \underline{x} \right|_{W,b} = \max\left\{\left| \underline{x}_1 \right|_W, \; p^{-b} \left| \underline{y} \right|_{W,b/p}^p, \;
\sup_{n \geq 0} \{p^{-bn} \left| x_{p^{-n}, p^n}\right| \} \right\}.
\]
It thus remains to prove that for $n \geq 0$,
\[
p^{-bn} \left| x_{p^{-n}, p^n}\right| \leq
\max\{\left| \underline{x}_1 \right|_W,  \left| \underline{y} \right|_{W,b/p}^p\}.
\]
We check this by induction on $n$, the case $n=0$ being obvious.
Given the claim for some $n$, 
let $\underline{z}_{p^{-n}} \in W_{p^{n+1}}(R)$ be the extension of $\underline{y}_{p^{-n}}$ by zero; thus $\underline{z}_{p^{-n}}$ is obtained from $\underline{x}_{p^{-n-1}}$ by replacing its final Witt component by zero. By 
Lemma~\ref{big Frobenius components lemma}, and using its notation, we have both
\begin{align*}
x_{p^{-n},p^n} &= x_{p^{-n-1},p^n}^p + px_{p^{-n-1},p^{n+1}} + f_{p^n}(x_{p^{-n-1},1}, \ldots, x_{p^{-n-1},p^n})\\
\intertext{and }
F(\underline{z}_{p^{-n}})_{p^n} &= x_{p^{-n-1},p^n}^p + p \cdot 0 + f_{p^n}(x_{p^{-n-1},1}, \ldots, x_{p^{-n-1},p^n}),\\
\intertext{so that }
x_{p^{-n-1},p^{n+1}} &= p^{-1}(x_{p^{-n},p^n} - F(\underline{z}_{p^{-n}})_{p^n} ).
\end{align*}
By the induction hypothesis, Proposition~\ref{Frobenius Verschiebung bound infinite case}(c), and the condition $b \geq 1$,
\begin{align*}
p^{-b(n+1)}
\left| 
x_{p^{-n-1},p^{n+1}}
\right|
&\leq p^{1-b(n+1)} \left| x_{p^{-n},p^n} - F(\underline{z}_{p^{-n}})_{p^n} \right| \\
&\leq p^{-bn} \max\{\left| x_{p^{-n},p^n} \right|,
\left| F(\underline{z}_{p^{-n}})_{p^n} \right|\} \\
&\leq \max\{p^{-bn} \left| x_{p^{-n},p^n} \right|,
p^{-bn} \left| \underline{z}_{p^{-n}} \right|^p_W\} \\
&\leq \max\{\left| \underline{x}_1 \right|_W, \left| \underline{y} \right|_{W,b/p}^p\}.
\end{align*}
This completes the induction and thus the proof.
\end{proof}

\begin{Thm} \label{compare to positive characteristic2}
For $0<b \leq 1$, the seminorm $\left| \bullet \right|_{W,b}$ is power-multiplicative;
for $b > 1$, it is equivalent to a power-multiplicative seminorm. Consequently, for all $b>0$, the subset $\Warrow^b(R)$ of $\Warrow(R)$ is in fact a subring.
\end{Thm}
\begin{proof}
We proceed by induction on the smallest nonnegative integer $h$ for which $p^h \geq b$. In the base case $h=0$, we have $b \leq 1$; in this case, we may deduce the claim from Theorem~\ref{compare to positive characteristic1} by embedding $A$ into a Witt-perfect ring (e.g., by repeatedly adjoining $p$-th roots of elements)
and then passing to the $p$-adic completion. 
%(Note that this is valid even if $A$ is not $p$-adically separated, since the intersection of the ideals $p^n A$ equals the kernel of the $p$-adic seminorm on $R$.)

To establish the induction hypothesis, it suffices to deduce the claim for $bp$ from the claim for $b$ whenever $bp \geq 1$.
The claim for $b$ is equivalent to the existence of $c_1, c_2 \geq 1$ such that for all  $\underline{x}, \underline{y} \in \Warrow^{b}(R)$,
\begin{align*}
\left| \underline{x} \underline{y} \right|_{W,b}
&\leq c_1 \left| \underline{x} \right|_{W,b}
\left| \underline{y} \right|_{W,b} \\
\left| \underline{x}^2 \right|_{W,b} 
&\geq c_2 \left| \underline{x} \right|^2_{W,b},
\end{align*}
because then $\left| \bullet \right|_{W,b}$ is equivalent to the power-multiplicative norm defined by $\lim \left| \underline{x}^n \right|_{W,b}^{1/n}$.
For $\underline{x},\underline{y} \in \Warrow^{pb}(R)$, by \eqref{equation for inverse Frobenius norm} we have
\begin{align*}
\left| \underline{x} \underline{y} \right|_{W,pb}
&\leq \max\{\left| \underline{x}_1 \underline{y}_1 \right|_W,
\left| F^{-1}(\underline{x} \underline{y}) \right|^p_{W,b} \}\\
&\leq \max\{\left| \underline{x}_1 \right|_W \left| \underline{y}_1 \right|_W,
c_1^p 
\left| F^{-1}(\underline{x}) \right|^p_{W,b}
\left| F^{-1}(\underline{y}) \right|^p_{W,b} \}\\
&\leq \max\{ \left| \underline{x}_1 \right|_W ,
c_1^p 
\left| F^{-1}(\underline{x}) \right|^p_{W,b} \} 
\max\{ \left| \underline{y}_1 \right|_W ,
c_1^p \left| F^{-1}(\underline{y}) \right|^p_{W,b} \} \\
&\leq 
c_1^{2p} p^{2pb} \max\{ \left| \underline{x}_1 \right|_W ,
p^{-pb}
\left| F^{-1}(\underline{x}) \right|^p_{W,b} \} 
\max\{ \left| \underline{y}_1 \right|_W ,
p^{-pb} \left| F^{-1}(\underline{y}) \right|^p_{W,b} \} \\
&\leq
c_1^{2p} p^{2pb} \left| \underline{x} \right|_{W,pb}
\left| \underline{y} \right|_{W,pb}
\end{align*}
and
\begin{align*}
\left| \underline{x}^2\right|_{W,pb}
&\geq \max\{\left| \underline{x}_1^2 \right|_W,
\left| F^{-1}(\underline{x}^2) \right|^p_{W,b} \}\\
&\geq \max\{\left| \underline{x}_1 \right|^2_W,
c_2^p \left| F^{-1}(\underline{x}) \right|^{2p}_{W,b} \}\\
&\geq c_2^p p^{2pb} \max\{\left| \underline{x}_1 \right|^2_W,
p^{-2pb} \left| F^{-1}(\underline{x}) \right|^{2p}_{W,b} \} \\
&\geq c_2^p p^{2pb} \left|\underline{x} \right|_{W,pb}^2.
\end{align*}
This completes the induction and hence the proof.
\end{proof}

\begin{Rmk}
Consider the specific case $A = \OCp$, $R = \Cp$, and $A' =: \widetilde{\mathbf{E}}^+, R' =: \widetilde{\mathbf{E}}$ the inverse limits under the $p$-power maps.  For $b>1$, although $\Warrow^b(\mathbf{E})$ has been studied previously, the ring $\Warrow^b(\Cp)$ does not seem to have been considered previously.
We cannot have an isometric isomorphism  $\Warrow^b(\mathbf{E}) \cong \Warrow^b(\Cp)$ for $b>1$ because the theta map $\Warrow^1(\mathbf{E}) \to \Cp$
does not extend to $\Warrow^b(\mathbf{E})$.
\end{Rmk}

\section{Finite \'etale extensions and overconvergent Witt vectors}
\label{sec:norms almost purity}

We now arrive at the main theorem of the paper, whose proof will occupy the entirety of this section.
\begin{Thm} \label{almost purity part 2}
Let $p$ be a prime. Let $A$ be a ring which is $p$-normal and Witt-perfect at $p$,
and put $R := A_p$. Let $S$ be a finite \'etale $R$-algebra
and let $B$ be the integral closure of $A$ in $S$. Equip $R$ and $S$ with the $p$-adic seminorms for the subrings $A$
and $B$, respectively.
Then for $b>0$, $\Warrow^b(S)$ is finite \'etale over $\Warrow^b(R)$. In particular, $\Warrow^\dagger(S)$ is finite \'etale over $\Warrow^\dagger(R)$.
\end{Thm}

We proceed using a sequence of lemmas.
\begin{Hyp}
For the remainder of \S\ref{sec:norms almost purity}, assume hypotheses as in Theorem~\ref{almost purity part 2}.
Also fix a sequence $x_1, x_2, \dots \in A$ as in Lemma~\ref{sequence of p-power roots}. 
\end{Hyp}

\begin{Lem} \label{density of subrings}
For $b' > b > 0$, $\Warrow^b(R)$ is dense in $\Warrow^{b'}(R)$
with respect to $\left| \bullet \right|_{W,b'}$.
\end{Lem}
\begin{proof}
We must check that for any $\underline{x} \in \Warrow^{b'}(R)$
and any $\epsilon > 0$, there exists $\underline{y} \in \Warrow^b(R)$
with $\left| \underline{x} - \underline{y} \right|_{W,b'} < \epsilon$.
To find $\underline{y}$, choose a positive integer $h$ for which
$p^{-b'n} \left| \underline{x}_{p^{-n}} \right|^{p^n}_{W} < \epsilon$
for all $n \geq h$. Then apply Lemma~\ref{Witt-perfect to norm-perfect} 
to lift $\underline{x}_{p^{-h}} \in W_{p^h}(R)$ to an element $\underline{y} \in \Warrow^b(R)$ with
$\left| \underline{y}_{p^{-n}} \right|_W^{p^{n}} \leq
\left| \underline{x}_{p^{-h}} \right|_W^{p^h}$ for all $n \geq h$.
This has the desired effect: for $n \leq h$ we have
$(\underline{x}- \underline{y})_{p^{-n}} = 0$,
while for $n > h$ we have
\[
p^{-b'n} \left| (\underline{x} - \underline{y})_{p^{-n}} \right|^{p^n}_W
\leq \max\{ p^{-b'n} \left| \underline{x}_{p^{-n}} \right|^{p^n}_W,
p^{-b'n} \left| \underline{y}_{p^{-n}} \right|^{p^n}_W\} < \epsilon.
\]
This proves the claim.
\end{proof}

\begin{Lem}\label{kernel norm}
For $j$ a positive integer, for $\underline{x}$ in the kernel of $F: W_{p^j}(S) \to
W_{p^{j-1}}(S)$, we have
\[
\left| w_1(\underline{x}) \right| = p^{-1/p - \cdots - 1/p^j} \left| \underline{x} \right|_{W}.
\]
\end{Lem}

\begin{proof}
Let $c = p^{-1/p - \cdots - 1/p^j}$.
By \cite[Corollary~2.6]{DK12}, if $\left| \underline{x} \right|_W < 1$, then $\left| w_1(\underline{x}) \right| \leq c$.  By multiplying by $[x_n^a p^b]$ for suitable $a,b,n$, we deduce that
for any $r \in \ZZ[p^{-1}]$, if $\left| \underline{x} \right|_{W} < p^r$ then
$\left| w_1(\underline{x}) \right| \leq cp^r$. 

Now assume $\left| w_1(\underline{x}) \right| < c$.   Note that because $S$ is $p$-torsion free, the element $\underline{x}$ is uniquely determined by $w_1(x)$.   By \cite[Corollary~2.6]{DK12}  again, we deduce $|\underline{x}|_W \leq 1$.  Multiplying by Teichm\"uller elements as above, we deduce that if $\left| w_1(\underline{x}) \right| < cp^r$ then $\left| \underline{x} \right|_{W} \leq p^r$.
Combining our two results, we deduce $\left| w_1(\underline{x}) \right| = c\left| \underline{x} \right|_{W}$, as desired.
\end{proof}

\begin{Lem} \label{lifting generators}
Fix positive integers $h$, $m$.
Choose $\underline{b}_1, \dots, \underline{b}_n \in \Warrow(B)$ 
for which
\begin{align*}
(x_{h}, p)B &\subseteq p^m B + w_{1}(\underline{b}_1) A + \cdots + w_{1}(\underline{b}_n) A.  \\
\intertext{Then for any $j \geq 0$, we have}
(x_{j+h}, p)B & \subseteq p^m B + w_{1/p^j}(\underline{b}_1) A + \cdots + w_{1/p^j}(\underline{b}_n) A.
\end{align*}
\end{Lem}
\begin{proof}
We proceed by induction on $j$, the base case $j=0$ being given.
Given the claim for some $j$, for any $b \in B$ we have 
\[
x_{j+h} b^p = p^m b_0 + a_1 w_{1/p^j}(\underline{b}_{1}) + \cdots + a_n w_{1/p^j}(\underline{b}_{n})
\]
for some $a_1,\dots,a_n \in A$, $b_0 \in B$. 
Since $A$ is Witt-perfect at $p$, we can find $a'_i \in A$
with $(a'_i)^p \equiv a_i \pmod{p}$; we then have
\begin{align*}
x_{j+h+1}^p b^p &\equiv (a'_1)^p \left( w_{1/p^{j+1}}(\underline{b}_{1})\right)^p  + \cdots + (a'_n)^p \left( w_{1/p^{j+1}}(\underline{b}_{n})\right)^p \mod pB \\
\intertext{and so if we write} 
x_{j+h+1} b &= a'_1 w_{1/p^{j+1}}(\underline{b}_{1}) + \cdots + a'_n  w_{1/p^{j+1}}(\underline{b}_{n}) + z,
\end{align*}
then by Lemma \ref{power ideals} we have that $z \in (x_1, p)B$.
In $B/p^m B$, both $x_1$ and $p$ are multiples of $x_{j+h+1}^p$; thus we can write 
\[
z \equiv x_{j+h+1}^{p-1} \left(x_{j+h+1}  z_0\right) \bmod p^mB.
\]
We may then run the argument again with $x_{j+h+1}  z_0$ in place of $x_{j+h+1}b$.  After suitably many repetitions, we find that
\[
x_{j+h+1} b \equiv a''_1 w_{1/p^{j+1}}(\underline{b}_{1}) + \cdots + a''_n  w_{1/p^{j+1}}(\underline{b}_{n}) \bmod p^mB.
\]
Since $p$ is a multiple of $x_{j+h+1}$ in $B/p^m B$, this completes the induction.
\end{proof}

\begin{Lem} \label{constructing generators}
For any $h \geq 1$, there exist elements $\underline{b}_1,\dots,\underline{b}_n \in \Warrow(B)$ for which
\[
(x_{j+h},p)B \subseteq w_{1/p^j}(\underline{b}_1) A + \cdots + w_{1/p^j}(\underline{b}_n) A
\qquad (j=0,1,\dots)
\]
\end{Lem}
\begin{proof}
Apply Theorem~\ref{almost purity part 1}
to construct elements $y_1,\dots,y_n \in B$ such that 
\[
(x_h,p)B \subseteq A y_1 + \cdots + A y_n.
\]
By Theorem~\ref{almost purity part 1} again,
$B$ is Witt-perfect at $p$, so we may lift $y_1,\dots,y_n$ to elements $\underline{y}_1,\dots,\underline{y}_n \in 
\Warrow(B)$.  These elements suffice to prove the $j = 0$ case of our desired result.  
We will construct additional elements $\underline{z}_1,\dots,\underline{z}_n \in \Warrow(B)$ (for the same choice of $n$)
so as to ensure that for all $j$,
\begin{equation} \label{constructing generators equation}
\begin{split}
(x_{j+h},p) B \subseteq w_{1/p^j}(\underline{y}_1) A + \cdots + w_{1/p^j}(\underline{y}_n) A 
& + \\ w_{1/p^j}&(\underline{z}_1) A  + \cdots + w_{1/p^j}(\underline{z}_n) A.
\end{split}
\end{equation}
To describe the elements $\underline{z}_i$, it is sufficient to specify the projection $\underline{z}_{ip^{-j}}$ of $\underline{z}_i$ to $W_{p^j}(B)$
for all $i$ and $j$. We will describe these projections recursively with respect to $j$.  For the case $j = 0$, we saw above that we may take $\underline{z}_{i1} = 0$ for all $i$.

Assume now that for some fixed $j$, we have $\underline{z}_{ip^{-j}} \in W_{p^j}(B)$ satisfying \eqref{constructing generators equation}.
Let $z_i$ be the first component of some lift of $\underline{z}_{ip^{-j}}$
to $W_{p^{j+1}}(B)$ under Frobenius. 
Since $z_i^{p^{j+1}} \equiv 0 \pmod{p}$, by Lemma \ref{power ideals} we have $z_i \in (x_{j+1},p)B \subseteq (x_{j+h+1},p)B$.  Returning momentarily to our elements $\underline{y}_i$ chosen earlier, 
by Lemma~\ref{lifting generators} we can find $v_i \in B$
so that 
\[
z_i + p^2 v_i \in w_{1/p^{j+1}}(\underline{y}_1) A + \cdots + w_{1/p^{j+1}}(\underline{y}_n) A.
\]

By \cite[Corollary~2.6]{DK12}, any element of $B$ congruent to $z_i$ modulo $p^2$
is the first Witt component of a lift of $\underline{z}_{ip^{-j}}$;
we may thus choose $\underline{z}_{ip^{-(j+1)}}$ so that 
\[
w_{1/p^{j+1}}(\underline{z}_i) = z_i + p^2 v_i + p^2 y_i
\]
and thus 
\[
p^2 y_i \in w_{1/p^{j+1}}(\underline{z}_i)A + w_{1/p^{j+1}}(\underline{y}_1) A + \cdots + w_{1/p^{j+1}}(\underline{y}_n) A .
\]

We know from Lemma~\ref{lifting generators} that 
\[
(x_{j+h + 1},p) B \subseteq p^3 B +  w_{1/p^{j+1}}(\underline{y}_1) A + \cdots + w_{1/p^{j+1}}(\underline{y}_n) A.
\]
To confirm \eqref{constructing generators equation}, it thus suffices to show that 
\[
p^3 B \subseteq w_{1/p^{j+1}}(\underline{y}_1) A + \cdots + w_{1/p^{j+1}}(\underline{y}_n) A 
+ w_{1/p^{j+1}}(\underline{z}_1) A  + \cdots + w_{1/p^{j+1}}(\underline{z}_n) A.
\]
We have already seen that the right side contains $p^2 y_i$ for every $i$.  Hence it suffices to show that
\[
p^3 B \subseteq p^2 y_1A + \cdots + p^2 y_nA.  
\]
On the other hand, the elements $y_1, \ldots, y_n$ were originally chosen so that 
\[
pB \subseteq y_1 A + \cdots + y_n A.
\]  
This completes the proof.
\end{proof}

\begin{Lem} \label{operator norm of projector}
With notation as in Lemma~\ref{constructing generators}, we have the following.
\begin{enumerate}
\item[(a)]
Any $s \in S$ can be expressed as $\sum_{i=1}^n r_i w_{1/p^j}(\underline{b}_i)$
for some $r_i \in R$
with $|r_i| \leq p^{1/p^{h+j}} |s|$.
\item[(b)]
Let $\pi: F \to S$ be the $R$-linear homomorphism with $F$ finite free with basis elements mapping to
$w_{1/p^j}(\underline{b}_1), \dots, w_{1/p^j}(\underline{b}_n)$.  There exists a splitting of $\pi$, denoted $\iota: S \rightarrow F$, depending on $j$ and $h$, such that $\iota$ has operator norm at most $p^{1/p^{h+j}}$. 
\end{enumerate}
\end{Lem}
\begin{proof}
We prove only (b), as this immediately implies (a).  The map $\pi$ is the localization of the $A$-module homomorphism $\pi_A: \oplus A \rightarrow B$ mapping basis elements to $w_{1/p^j}(\underline{b}_1), \dots, w_{1/p^j}(\underline{b}_n)$.  By Lemma~\ref{constructing generators}, the cokernel of $\pi_A$ is annihilated by $(x_{j+h},p)$.  Then by 
Theorem~\ref{almost purity part 1}(c) and localization, there exists an $R$-module homomorphism 
\[
\iota_{x_{j+h}}: S \rightarrow F
\]
such that the composition $\iota_{x_{j+h}} \circ \pi$ is multiplication by $x_{j+h}$ and such that $\iota_{x_{j+h}}$ has operator norm at most 1.  By the same reasoning, there exists
\[
\iota_p: B \rightarrow \oplus A
\]
such that $\iota_p \circ \pi$ is multiplication by $p$ and $\iota_p$ has operator norm at most 1. 

By the definition of the element $x_{j+h}$, there exists $v \in A$ such that 
\[
p = x_{j+h}^{p^{j+h}} + p^2 v.
\]
Now define
\[
\iota := \left(\iota_{x_{j+h}} x_{j+h}^{p^{j+h}-1} + \iota_p pv\right) p^{-1}.
\]
Then because $\iota_{x_{j+h}} x_{j+h}^{p^{j+h}-1} $ has operator norm at most $p^{-1 + 1/p^{j+h}}$
and $\iota_p pv$ has operator norm at most $p^{-1}$, we deduce that $\iota$ has operator norm at most $p^{1/p^{j+h}}$, as required.
\end{proof}

\begin{Prop} \label{projective if p-adically complete}
In addition to the hypotheses of Theorem~\ref{almost purity part 2}, assume furthermore that $b \leq 1$ and that $A$ is $p$-adically complete.  Then $\Warrow^b(S)$ is a finite projective $\Warrow^b(R)$-module.
\end{Prop}

\begin{proof}
Write $R' := \varprojlim R$ and $S' := \varprojlim S$ for the inverse systems in which the transition maps are the $p$-power map. 
By Theorem~\ref{compare to positive characteristic1},
we may identify $\Warrow^b(R)$ and $\Warrow^b(S)$ with
$\Warrow^b(R')$ and $\Warrow^b(S')$; the latter coincide with the rings
$\mathcal{R}^{\text{int},\frac{1}{b}}_{R'}$ 
and $\mathcal{R}^{\text{int},\frac{1}{b}}_{S'}$ defined in
\cite[Definition~5.1.1]{KL11}. Hence we may finish by invoking \cite[Proposition~5.5.4]{KL11}.  
\end{proof}

\begin{Rmk}
Because of our goal to adapt these results to a global setting, it would be desirable to have a proof of the preceding result that does not pass to characteristic~$p$.  
\end{Rmk}

The core part of the proof of Theorem~\ref{almost purity part 2} is the proof that $\Warrow^b(S)$ is a finite $\Warrow^b(R)$-module;
we address this point now.

\begin{Prop} \label{construct surjective map}
Fix a real number $b > 0$.  Choose a positive integer $h$ such that $p^{-h} < b$, and with respect to this $h$, 
choose $\underline{b}_1, \dots, \underline{b}_n \in \Warrow(B)$ 
as in Lemma~\ref{constructing generators}. 
Let $\pi: G \to \Warrow^b(S)$ be the $\Warrow^b(R)$-linear homomorphism
with $G$ finite free with basis elements mapping to $\underline{b}_1, \dots, \underline{b}_n$.
Then the map $\pi: G \to \Warrow^b(S)$ is surjective. 
\end{Prop}

\begin{proof}
Choose $\underline{x} \in \Warrow^b(S)$, and for all $j$, assume $D_j$ is such that $p^{-bj} \left| \underline{x}_{p^{-j}}\right|_W^{p^j} \leq D_j$; by the definition of $\Warrow^b(S)$, we may choose $D_j$ so that $D_j \rightarrow 0$ as $j$ increases.  After increasing some of the constants $D_j$, we may assume that, in addition to approaching 0, they also satisfy
\[
\frac{1}{p^{b-p^{-h}}} D_j \leq D_{j+1} \leq D_j.
\]
We will produce coefficients $\underline{r}_1, \ldots, \underline{r}_n \in \Warrow^b(R)$ such that 
\[
p^{-b(j+1)} \left| \underline{r}_{ip^{-j}}\right|_W^{p^j} \leq D_j \qquad (i = 1, \ldots, n)
\]
 and
\[
\underline{x} = \underline{r}_1 \underline{b}_1 + \cdots + \underline{r}_n \underline{b}_n.
\]
It will be convenient to note the following explicit bounds on the Witt vector norms:
\[
\left| \underline{x}_{p^{-j}}\right|_W \leq D_j^{\frac{1}{p^j}}p^{\frac{bj}{p^j}}, \text{ and we want to prove } \left| \underline{r}_{ip^{-j}}\right|_W \leq D_j^{\frac{1}{p^j}}p^{\frac{b(j+1)}{p^j}}.
\]
To completely specify an element $\underline{r}_i \in \Warrow^b(R)$, it suffices to specify its projections $\underline{r}_{ip^{-j}} \in W_{p^j}(R)$ for each $j \geq 0$.  

We first use Lemma~\ref{operator norm of projector} to treat the base case $j = 0$.  This lemma guarantees that we can find $\underline{r}_{i1} \in W_1(R) \cong R$ such that
\[
\underline{x}_1 = \underline{r}_{11} \underline{b}_{11} + \cdots + \underline{r}_{n1} \underline{b}_{n1} 
\]
and such that, furthermore, 
\[
\left|\underline{r}_{i1}\right|_W \leq p^{\frac{1}{p^h}} \left|\underline{x}_1\right|_W \leq p^b D_0 \qquad (i = 1, \ldots, n).
\]
This completes the base case of the induction.

Now inductively assume we have found elements $\underline{r}_{1p^{-j}}, \ldots, \underline{r}_{np^{-j}} \in W_{p^j}(R)$ such that
\[
\underline{x}_{p^{-j}} = \underline{r}_{1p^{-j}} \underline{b}_{1p^{-j}} + \cdots + \underline{r}_{np^{-j}} \underline{b}_{np^{-j}}
\]
and such that, furthermore,
\[
\left| \underline{r}_{ip^{-j}}\right|_W \leq D_j^{\frac{1}{p^j}}p^{\frac{b(j+1)}{p^j}}. \qquad (i = 1, \ldots, n).
\]
We will now construct elements $\underline{r}_{1p^{-(j+1)}}, \ldots, \underline{r}_{np^{-(j+1)}}$ satisfying the analogous conditions. 

As preliminary values, apply Lemma~\ref{Witt-perfect to norm-perfect} to find elements $\underline{r}'_{1p^{-(j+1)}}, \ldots, \underline{r}'_{np^{-(j+1)}}$ such that 
\[
F\left( \underline{r}'_{ip^{-(j+1)}}\right) = \underline{r}_{ip^{-j}} \qquad (i = 1, \ldots, n),
\]
and such that, furthermore
\begin{equation} \label{prelim r norm}
\left|\underline{r}'_{ip^{-(j+1)}}\right|_W^p \leq \left|\underline{r}_{ip^{-j}}\right|_W \leq D_j^{\frac{1}{p^j}}p^{\frac{b(j+1)}{p^j}}.
\end{equation}
The difference
\[
\underline{u} := \underline{x}_{p^{-(j+1)}} - \sum_{i=1}^n \underline{r}'_{ip^{-(j+1)}} \underline{b}_{ip^{-(j+1)}},
\]
is then an element of the kernel of $F: W_{p^{j+1}}(S) \rightarrow W_{p^j}(S)$.  Let $u$ be the first component of this element, and apply Lemma~\ref{operator norm of projector} to write
\[
u = u_1 w_{1/p^{j+1}}(\underline{b}_1) + \cdots + u_n w_{1/p^{j+1}}(\underline{b}_n),
\]
with 
\[
|u_i| \leq p^{\frac{1}{p^{h+j+1}}}|u| \qquad (i = 1, \ldots, n).
\]
Then extend each $u_i$ into an element $\underline{u}_i$ in the kernel of $F: W_{p^{j+1}}(R) \rightarrow W_{p^j}(R)$.  For $c := p^{-1/p - \cdots - 1/p^{j+1}}$, by Lemma~\ref{kernel norm}, we have
\[
\left|u\right| = c\left| \underline{u} \right|_W \qquad \text{ and } \qquad |u_i| = c \left|\underline{u}_i\right|_W \qquad (i = 1, \ldots, n).
\]

Now set $\underline{r}_{ip^{-(j+1)}} = \underline{r}'_{ip^{-(j+1)}} + \underline{u}_i$ for each $i = 1, \ldots, n$.  Then we have 
\[
\underline{x}_{p^{-(j+1)}} = \underline{r}_{1p^{-(j+1)}} \underline{b}_{1p^{-(j+1)}} + \cdots + \underline{r}_{np^{-(j+1)}} \underline{b}_{np^{-(j+1)}} 
\]
because both sides are elements in $W_{p^{j+1}}(S)$ with the same first coordinate and with the same image under $F$ in $W_{p^j}(S)$.  It remains to consider the norms on these coefficients.  To complete the induction, we wish to show that 
\[
\left| \underline{r}_{ip^{-(j+1)}}\right|_W \leq D_{j+1}^{\frac{1}{p^{j+1}}}p^{\frac{b(j+2)}{p^{j+1}}} \qquad (i = 1, \ldots, n).
\]
Because $\underline{r}_{ip^{-(j+1)}} = \underline{r}'_{ip^{-(j+1)}} + \underline{u}_i$, it suffices to treat these two terms separately. 

By (\ref{prelim r norm}), we know 
\[
\left| \underline{r}'_{ip^{-(j+1)}} \right|_{W} \leq D_j^{\frac{1}{p^{j+1}}}p^{\frac{b(j+1)}{p^{j+1}}},
\]
and so the result follows from our assumption that $D_j \leq p^{b - p^{-h}}D_{j+1} \leq p^b D_{j+1}$.  
We now treat the term $\underline{u}_i$ similarly.  Recall that $u$ and $u_i$ denote the first Witt coordinates of $\underline{u}$ and $\underline{u}_i$, respectively.  We know
\[
\left| \underline{u}_i \right|_W = c^{-1} \left|u_i\right| \leq c^{-1} p^{1/p^{h+j+1}} |u| = p^{1/p^{h+j+1}} \left| \underline{u} \right|_W.  
\]
Recalling the definition of $\underline{u}$, we reduce in this way to proving 
\[
\left| \underline{x}_{p^{-(j+1)}} - \sum_{i=1}^n \underline{r}'_{ip^{-(j+1)}} \underline{b}_{ip^{-(j+1)}} \right|_W \leq p^{\frac{-1}{p^{h+j+1}}} D_{j+1}^{\frac{1}{p^{j+1}}}p^{\frac{b(j+2)}{p^{j+1}}} 
\]
With regards to the $\underline{x}_{p^{-(j+1)}}$ term, this reduces to showing
\[
D_{j+1}^{\frac{1}{p^{j+1}}}p^{\frac{b(j+1)}{p^{j+1}}} \leq p^{\frac{-1}{p^{h+j+1}}} D_{j+1}^{\frac{1}{p^{j+1}}}p^{\frac{b(j+2)}{p^{j+1}}},
\]
which follows from our bound on $h$.

We now must show that for any $i$, we have
\[
\left| \underline{r}'_{ip^{-(j+1)}} \underline{b}_{ip^{-(j+1)}} \right|_W \leq p^{\frac{-1}{p^{h+j+1}}} D_{j+1}^{\frac{1}{p^{j+1}}}p^{\frac{b(j+2)}{p^{j+1}}}.
\]
Because $\underline{b}_i \in \Warrow(B) \subseteq \Warrow^b(S)$, and hence its norm is at most 1, and because $| \bullet |_W$ is submultiplicative, it suffices to treat the $\underline{r}'_i$ term.  By construction and the inductive hypothesis, we have
\[
\left| \underline{r}'_{ip^{-(j+1)}} \right|_W \leq \left|  \underline{r}_{ip^{-j}} \right|_W^{1/p}  \leq D_j^{\frac{1}{p^{j+1}}} p^{\frac{b(j+1)}{p^{j+1}}},
\]
so we are again finished from $D_j \leq p^{b - p^{-h}} D_{j+1}$.  
This completes the induction, and hence the proof that $\pi: G \rightarrow \Warrow^b(S)$ is surjective.  
\end{proof}

\begin{Cor} \label{same generators}
For any $b > 0$, the natural map
\[
\Warrow^b( \widehat{R} ) \otimes_{\Warrow^b(R)} \Warrow^b(S) \rightarrow \Warrow^b(\widehat{S})
\]
is surjective.  
\end{Cor}

\begin{proof}
By Proposition~\ref{construct surjective map}, $\Warrow^b(\widehat{S})$ is a finite $\Warrow^b(\widehat{R})$-module, with generators chosen as in Lemma~\ref{constructing generators}.  On the other hand, if we choose generators as in Lemma~\ref{constructing generators} for $B$ instead of $\widehat{B}$, it is clear that they also satisfy the hypotheses for $\widehat{B}$.  These generators then lie in the image of $\Warrow^b( \widehat{R} ) \otimes_{\Warrow^b(R)} \Warrow^b(S) \rightarrow \Warrow^b(\widehat{S})$, and hence the map is indeed surjective.  
\end{proof}

We will use Proposition~\ref{projective if p-adically complete} concerning $p$-adically complete rings to deduce the following.  It says that in the case $b \leq 1$, if we can prove that $\Warrow^b(S)$ is a finite $\Warrow^b(R)$-module, then we get that $\Warrow^b(S)$ is a finite \emph{projective} $\Warrow^b(R)$-module for free.

\begin{Thm} \label{finite implies projective} 
Suppose that $b \leq 1$.
Let $G$ denote a finite free $\Warrow^b(R)$-module and assume there exists a surjective $\Warrow^b(R)$-module homomorphism $\pi: G \twoheadrightarrow \Warrow^b(S)$.  Then there exists a $\Warrow^b(R)$-linear splitting $\iota: \Warrow^b(S) \rightarrow G$ of $\pi$.  
\end{Thm}

\begin{proof}
Write $\hat{R}$ for the $p$-adic completion of $R$, and similarly for other rings; for example, $\hat{G}$ denotes a finite free $\Warrow^b(\hat{R})$-module.  By Corollary~\ref{same generators}, base extension of $\pi$ produces a surjection 
\[
\hat{\pi}: \hat{G} \twoheadrightarrow \Warrow^b(\hat{S}),
\]
and by Proposition~\ref{projective if p-adically complete} there is a splitting $\hat{\iota}$ of $\hat{\pi}$.  Our strategy is to perturb $\hat{\iota}$ slightly so that it is the base extension of our desired splitting $\iota$.

By base extension of $\pi$, we obtain surjections
\[
\hat{G} \to \Warrow^b(\hat{S}), \qquad
G_n \to W_{p^n}(S), \qquad
\hat{G}_n \to W_{p^n}(\hat{S}).
\]
Let $\hat{K}, K_n, \hat{K}_n$ be the kernels of these maps. Let $\hat{T}, T_n, \hat{T}_n$ be the spaces of splitting of these maps; these are torsors for $\hat{H}, H_n, \hat{H}_n$, where
\[
\hat{H} := \Hom_{\Warrow^b(\hat{R})}(\Warrow^b(\hat{S}), \hat{K})
\]
and so on. Note that $\hat{K},K_n,\hat{K}_n$ and $\hat{H}, H_n, \hat{H}_n$ are finite projective modules over their respective rings.
(For $\hat{K}$ and $\hat{H}$, we are using Proposition~\ref{projective if p-adically complete}
in order to know that $\Warrow^b(\hat{S})$ is a finite projective module over $\Warrow^b(\hat{R})$; this is why we must assume that $b \leq 1$.)

For everything in sight, 
Frobenius gives us maps from $*_{n+1}$ to $*_{n}$.
More precisely, we have isomorphisms
\[
*_{n+1} \otimes_{W_{p^{n+1}},F} W_{p^n} \cong *_n.
\]
This even makes sense for the $T$'s: if we take a splitting of $T_{n+1}$, use it to identify $T_{n+1}$ with $H_{n+1}$, then map the same splitting to $T_n$ and use it to identify $T_n$ with $H_n$, we can transfer the identification $H_{n+1} \otimes_{W_{p^{n+1}}(R),F} W_{p^n}(R) \cong H_n$ to get an identification of ``$T_{n+1} \otimes_{W_{p^{n+1}}(R),F} W_{p^n}(R)$'' with $T_n$. 

Fix a norm on $\hat{H}$ by viewing $\hat{H}$ as a finite projective module over $\Warrow^b(\hat{R})$ and applying the norm on $\Warrow^b(\hat{R})$; the equivalence class of this norm does not depend on any choices (see \cite[Lemma~2.2.12]{KL11}).
We now define a sequence of splittings 
\[
\hat{t}^{(-1)},\, \hat{t}^{(0)},\, \dots \in \hat{T}
\] 
which will converge to an element $\hat{t} \in \hat{T}$ which is the base extension of a splitting 
\[
\iota: \Warrow^b(S) \rightarrow G.
\] 

We define the splittings $\hat{t}^{(i)}$ as follows. To begin with, take $\hat{t}^{(-1)} = \hat{\iota} \in \hat{T}.$  Now suppose we have been given $\hat{t}^{(n-1)} \in \hat{T}$ for some $n \geq 0$.  If $n \geq 1$, assume further that the image $\hat{t}^{(n-1)}_{n-1} \in \hat{T}_{n-1}$ of $\hat{t}^{(n-1)}$
is the base extension of some $t_{n-1} \in T_{n-1}$.  We would like to choose $\hat{t}^{(n)} \in \hat{T}$ so that:
\begin{enumerate}
\item[(a)]
$\hat{t}^{(n)}$ and $\hat{t}^{(n-1)}$ have the same image in $\hat{T}_{n-1}$
(that is, $\hat{t}^{(n)} - \hat{t}^{(n-1)}$ maps to zero in $\hat{H}_{n-1}$);
\item[(b)]
the norm of $\hat{t}^{(n)} - \hat{t}^{(n-1)}$ in $\hat{H}$ is at most $p^{-n}$;
\item[(c)]
the image of $\hat{t}^{(n)}$ in $\hat{T}_n$ is the base extension of some $t_n \in T_n$.
\end{enumerate}
We momentarily postpone verifying that we can produce such a $\hat{t}^{(n)} \in \hat{T}$ and instead explain why this will complete the proof. If these conditions are all satisfied, then the elements $t_n$ from condition (c) can be taken to be Frobenius compatible by condition (a), in which case they produce a splitting 
\[
\iota: \Warrow(S) \rightarrow \bigoplus \Warrow(R)
\]
without any overconvergence condition. 
On the other hand, by (b) the $\hat{t}_n$ converge in $\hat{T}$ to  a splitting $\hat{\iota}$
of $\hat{\pi}$, so $\iota$ maps $\Warrow^b(S)$ into
the intersection of $\bigoplus \Warrow(R)$ and $\hat{G}$
within $\bigoplus \Warrow(\hat{R})$. This intersection is precisely $G$,
so we obtain a splitting of $\pi$ as desired.

 Thus it remains to show how to produce the desired element $\hat{t}^{(n)} \in \hat{T}$ from the given elements $\hat{t}^{(n-1)} \in \hat{T}$ and $t_{n-1} \in T_{n-1}$.
Let $U_n$ be the set of lifts of $t_{n-1}$ to $T_n$.
Fixing a norm on $H_n$, the induced norm on $\hat{H}_n$ will be equivalent to the one induced on $\hat{H}_n$ by $\hat{H}$
(again see \cite[Lemma~2.2.12]{KL11}). In this sense, $U_n$ is dense in the set of lifts of $\hat{t}^{(n-1)}_{n-1}$ from $\hat{T}_{n-1}$ to $\hat{T}_n$, so we can pick $t_n \in T_n$ which maps to $t_{n-1}$ and whose base extension to $\hat{T}_n$ lifts to $\hat{t}^{(n)} \in \hat{T}$ of the right form. 
\end{proof}

\begin{Lem} \label{finite projective module}
For all $b$, $\Warrow^b(S)$ is a finite projective $\Warrow^b(R)$-module. 
\end{Lem}

\begin{proof}
For $b \leq 1$ this follows from Theorem~\ref{finite implies projective} and Proposition~\ref{construct surjective map}. Now fix $b > 1$. From the $b = 1$ case, we may produce a surjection $\pi:  \oplus \Warrow^1(R) \rightarrow \Warrow^1(S)$ and
a splitting $\iota: \Warrow^1(S) \to \oplus \Warrow^1(R)$ of $\Warrow^1(R)$-modules.  (Notice that the elements $\underline{b}_1, \ldots, \underline{b}_n$ from the statement of Proposition~\ref{construct surjective map} also satisfy the hypotheses of that proposition for $b > 1$.  Thus the base change of $\pi$ to $\oplus \Warrow^b(R) \rightarrow \Warrow^b(S)$ is also surjective.  We will use this below.)  The construction of $\iota$ in the proof of Theorem~\ref{finite implies projective} is constructed level-wise, i.e., the projection of $\iota(\underline{x})$ to $\oplus W_{p^n}(R)$ depends only on the projection of $\underline{x}$ to $\underline{x}_{p^{-n}} \in W_{p^n}(S)$.  Thus we may view $\iota$ as the restriction to $\Warrow^1(S)$ of another map $\Warrow(S) \rightarrow \oplus \Warrow(R)$ with no overconvergence conditions.  We also denote this latter map by $\iota$.  We claim that the restriction of $\iota$ to $\Warrow^b(S)$ is a splitting of $\oplus \Warrow^b(R) \rightarrow \Warrow^b(S)$; to prove this, we only need to show that $\iota( \Warrow^b(S) ) \subseteq \oplus \Warrow^b(R)$.  
%This is a statement concerning only norms of elements in $R$ and $S$, so it suffices to consider this after passing to the $p$-adic completions.  (Note that this assertion holds without any assumption of $p$-adic separatedness, because the $p$-adic norm of an element does not change if .)  

%We next prove the following claim.  Although the norm $\left| \bullet \right|_{W,b}$ is not finite on all of $\Warrow(R)$ and $\Warrow(S)$, the $\Warrow(R)$-module homomorphism $\iota$ is continuous with respect to $\left| \bullet \right|_{W,b}$ in the sense that for $\underline{x}^{(1)}, \underline{x}^{(2)}, \ldots \in \Warrow(S)$ with $\left| \underline{x}^{(i)} \right|_{W,b}$ approaching zero, then $\left| \iota(\underline{x}^{(i)}) \right|_{W,b}$ approaches zero.  (We define $\left| \bullet \right|_{W,b}$ on $\oplus \Warrow(R)$ by taking the maximum norm across all components in the direct sum.) 
Consider the map 
\[
\pi: \oplus \Warrow(R) \rightarrow \Warrow(S)
\]
defined by sending $e_i \mapsto \underline{b}_i$, with $\underline{b}_i$ as above.  Because the elements $\underline{b}_i$ all lie in $\Warrow(B)$, they in particular lie in $\Warrow^b(S)$, and so $\iota(\underline{b}_i) = \iota(\pi(e_i)) \in \oplus \Warrow^b(R)$ for all $i$.  Because $\iota$ and $\pi$ are both $\Warrow(R)$-module homomorphisms, and because $\Warrow^b(R)$ is closed under multiplication, we deduce that $\iota(\pi( \oplus \Warrow^b(R))) \subseteq \oplus \Warrow^b(R)$. On the other hand, as noted parenthetically above, $\pi(\oplus \Warrow^b(R)) = \Warrow^b(S)$.  Combining these last two observations, we have just checked $\iota(\Warrow^b(S)) \subseteq \oplus \Warrow^b(R)$, as required.  
\end{proof}

%\begin{Lem} \label{finite projective module}
%For all $b$, $\Warrow^b(S)$ is a finite projective $\Warrow^b(R)$-module. 
%\end{Lem}
%\begin{proof}
%For $b \leq 1$ this follows from Theorem~\ref{finite implies projective} and Proposition~\ref{construct surjective map}. For $b > 1$, we reduce to the case $b=1$ as follows.
%Apply Proposition~\ref{construct surjective map} with $b=1$ to produce
%a surjection $\pi: G \rightarrow \Warrow^1(S)$ and
%a splitting $\iota: \Warrow^1(S) \to G$ of $\Warrow^1(R)$-modules.
%By Proposition~\ref{construct surjective map} again, the induced map
%\[
%G \otimes_{\Warrow^1(R)} \Warrow^b(R) \to \Warrow^b(S)
%\]
%is surjective. By Lemma~\ref{density of subrings}, $\Warrow^1(S)$ is dense in $\Warrow^b(S)$.  We can then extend the map $\iota$ to a map $\Warrow^b(S) \rightarrow G \otimes_{\Warrow^1(R)} \Warrow(R)$; we must check that the image lies in $G \otimes_{\Warrow^1(R)} \Warrow^b(R)$. We may check this after passing to the $p$-adic completions $\widehat{R}$ and $\widehat{S}$, and in this case we are finished by Lemma~\ref{complete Witt vectors}. This proves the claim in the case $b>1$.
%\end{proof}

\begin{Lem} \label{compatibility with tensor product}
For $S'$ another finite \'etale algebra, the natural map
\[
\Warrow^b(S) \otimes_{\Warrow^b(R)} \Warrow^b(S') \to \Warrow^b(S \otimes_R S')
\]
is an isomorphism.
\end{Lem}
\begin{proof}
Apply Proposition~\ref{construct surjective map} to construct surjections
$\pi: G \to \Warrow^b(S)$ and $\pi': G' \to \Warrow^b(S')$.
By Proposition~\ref{construct surjective map} again, the induced map
\[
\pi \otimes \pi': G \otimes_{\Warrow^b(R)} G' \to \Warrow^b(S \otimes_R S')
\]
is also surjective. Consequently, the natural map
\[
\Warrow^b(S) \otimes_{\Warrow^b(R)} \Warrow^b(S') \to \Warrow^b(S \otimes_R S')
\]
is a surjective $\Warrow^b(R)$-linear homomorphism between finite projective modules of everywhere equal rank. This map is thus forced to be an isomorphism by Nakayama's lemma.
\end{proof}

\begin{proof}[Proof of Theorem~\ref{almost purity part 2}]
Since $S$ is finite \'etale as an $R$-algebra, it is finite projective as a module over both $R$ and $S \otimes_R S$ (via the multiplication map $S \otimes_R S \to S$). 
By Lemma~\ref{finite projective module}, 
$\Warrow^b(S)$ is finite projective as a module over both $\Warrow^b(R)$
and $\Warrow^b(S \otimes_R S)$. However, by Lemma~\ref{compatibility with tensor product} the latter ring is isomorphic to $\Warrow^b(S) \otimes_{\Warrow^b(R)} \Warrow^b(S)$. Therefore, $\Warrow^b(S)$ is a finite \'etale $\Warrow^b(R)$-algebra, as desired.
\end{proof}

\section{Associating $(\varphi^{-1}, \Gamma)$-modules to Artin motives} \label{sec:phi,Gamma}

We conclude with a baby step in the direction of global $(\varphi, \Gamma)$-modules in order to illustrate the motivation for the preceding constructions.

Let $L$ denote a finite extension of $\QQ_p$.  The comparison isomorphism in $p$-adic Hodge theory associates to each smooth and proper variety over $L$ a certain algebraic object called a \emph{$(\varphi, \Gamma)$-module}. Such an object is a module over a complicated base ring equipped with some additional structures. 
More precisely, $(\varphi, \Gamma)$-modules can be associated to arbitrary $p$-adic representations of the absolute Galois group $G_L$, and the one associated to an algebraic variety is obtained from the $p$-adic \'etale cohomology.

A hypothetical theory of \emph{global $(\varphi, \Gamma)$-modules} would similarly associate an algebraic object to a smooth proper variety over $\QQ$ and would be related to \'etale cohomology as a representation of $G_{\QQ}$ rather than $G_L$. While we are quite far at present from any such theory, we can at least treat the case of zero-dimensional varieties, which correspond to discrete $G_{\QQ}$-representations. This case is essentially a formal consequence of Theorem~\ref{almost purity part 2}.

\begin{Hyp}
Throughout the remainder of this paper, we consider $\Qab$ equipped with the supremum norm taken over all places above $p$.  The results here would hold similarly for $\QQ(\mu_{p^{\infty}})$, or for an intermediate field between $\QQ(\mu_{p^{\infty}})$ and $\Qab$, but we work with $\Qab$ to emphasize our eventual goal of having a theory in which there is no distinguished prime $p$.  In that theory, our supremum norm would be replaced by the supremum over all places (including perhaps the infinite places), and our $p$-typical Witt vectors would be replaced by big Witt vectors.
\end{Hyp}

\begin{Def}
An \emph{Artin motive} over $\QQ$ with coefficients in a field $K$ is a finite dimensional $K$-vector space equipped with a discrete action of $G_{\QQ}$.
\end{Def}

In our analogue of the theory of $(\varphi, \Gamma)$-modules, Artin motives over $\QQ$ with coefficients in $K$ will play the role of the $p$-adic \'etale cohomology groups associated to varieties over a $p$-adic field.  The role of $(\varphi, \Gamma)$-modules will be played by the following.

\begin{Def}
A \emph{$(\varphi^{-1}, \Gamma)$-module} over the ring $\Warrow^{\dag}(\Qab) \otimes_{\ZZ} K$ is a finite projective module over $\Warrow^{\dag}(\Qab) \otimes_{\ZZ} K$ equipped with semilinear actions of the inverse Frobenius $\varphi^{-1}$ and the group $\Gamma := \Gal(\Qab/\QQ) \cong \widehat{\ZZ}^{\times}$.
\end{Def}

\begin{Rmk}
The actions of $\varphi^{-1}$ and $\Gamma$ on the base ring $\Warrow^{\dag}(\Qab) \otimes_{\ZZ} K$ are induced by the actions on the left-hand factor.  Note that we must use $\varphi^{-1}$ instead of $\varphi$ because there is no action of $\varphi$ on $\Warrow^\dagger(\Qab)$.
\end{Rmk}

We now describe how to associate a $(\varphi^{-1}, \Gamma)$-module over $\Warrow^{\dag}(\Qab) \otimes_{\ZZ} K$ to an Artin motive.  Let $V$ denote an Artin motive over $\QQ$ with coefficients in $K$.  Because the action of $G_{\QQ}$ is by definition \emph{discrete}, we can find a finite Galois extension $F/\QQ$ such that the Galois action on $V$ factors through $\Gal(F/\QQ)$.  To $V$, we then associate the $\Warrow^{\dag}(\Qab) \otimes_{\ZZ} K$-module
\[
\left(\Warrow^{\dag}(F \otimes_{\QQ} \Qab) \otimes_{\ZZ} V\right)^{\Gal(F/\QQ)}.
\]

\begin{Thm} \label{Artin motive theorem}
With notation as above,
\[
\left(\Warrow^{\dag}(F \otimes_{\QQ} \Qab) \otimes_{\ZZ} V\right)^{\Gal(F/\QQ)}
\]
is a finite projective $\Warrow^{\dag}(\Qab) \otimes_{\ZZ} K$-module.
\end{Thm}

\begin{proof}

We will describe how this follows from Theorem~\ref{almost purity part 2}. We will use without further comment the compatibility of the functor $F \mapsto \Warrow^\dagger(F \otimes_{\QQ} \Qab)$ on finite \'etale $\QQ$-algebras with the formation of tensor products; see
Lemma~\ref{compatibility with tensor product}.

 We first observe that $\Warrow^{\dag}(F \otimes_{\QQ} \Qab)$ is faithfully flat over $\Warrow^{\dag}(\Qab)$.  Then by faithfully flat descent \cite[Tag~058S]{stacks-project}, to deduce the result of the lemma, it suffices to show that
\begin{equation} \label{22113 eqn}
\left( \Warrow^{\dag}(F \otimes_{\QQ} \Qab)\otimes_{\ZZ} K \right) \otimes_{\Warrow^{\dag}(\Qab) \otimes_{\ZZ} K}\left(\Warrow^{\dag}(F \otimes_{\QQ} \Qab) \otimes_{\ZZ} V\right)^{\Gal(F/\QQ)}
\end{equation}
is finite projective over $\Warrow^{\dag}(F \otimes_{\QQ} \Qab)\otimes_{\ZZ} K$.  To prove this, we will show that the module in (\ref{22113 eqn}) is isomorphic to $\Warrow^{\dag}(F \otimes_{\QQ} \Qab)\otimes_{\ZZ} V$.  To prove they are isomorphic, consider the following isomorphisms:
\begin{align*}
\left( \Warrow^{\dag}(F \otimes_{\QQ} \Qab)\otimes_{\ZZ} K \right) &\otimes_{\Warrow^{\dag}(\Qab) \otimes_{\ZZ} K}\left(\Warrow^{\dag}(F \otimes_{\QQ} \Qab) \otimes_{\ZZ} V\right)^{\Gal(F/\QQ)} \\
&\cong \left[ \left( \Warrow^{\dag}(F \otimes_{\QQ} \Qab)\otimes_{\ZZ} K \right) \otimes_{\Warrow^{\dag}(\Qab) \otimes_{\ZZ} K}\left(\Warrow^{\dag}(F \otimes_{\QQ} \Qab) \otimes_{\ZZ} V \right) \right]^{\Gal(F/\QQ)} \\
\intertext{Here we declare that $\Gal(F/\QQ)$ acts trivially on the left factor.  The isomorphism follows from projectivity of $\Warrow^{\dag}(F \otimes_{\QQ} \Qab)$ over $\Warrow^{\dag}(\Qab)$.}
&\cong \left[ \Warrow^{\dag}(F \otimes_{\QQ} \Qab) \otimes_{\Warrow^{\dag}(\Qab)} \Warrow^{\dag}(F \otimes_{\QQ} \Qab) \otimes_{\ZZ} V \right]^{\Gal(F/\QQ)} \\
&\cong \left[ \Warrow^{\dag}\left((F \otimes_{\QQ} \Qab) \otimes_{\Qab} (F \otimes_{\QQ} \Qab) \right) \otimes_{\ZZ} V \right]^{\Gal(F/\QQ)} \\
\intertext{To see this last isomorphism, note that $\Warrow^{\dag}(F \otimes_{\QQ} \Qab) \otimes_{\Warrow^{\dag}(\Qab)} \Warrow^{\dag}(F \otimes_{\QQ} \Qab)$ and $\Warrow^{\dag}\left((F \otimes_{\QQ} \Qab) \otimes_{\Qab} (F \otimes_{\QQ} \Qab) \right) $ are both finite \'etale algebras over $\Warrow^{\dag}(F \otimes_{\QQ} \Qab)$ lifting $(F \otimes_{\QQ} \Qab) \otimes_{\Qab} (F \otimes_{\QQ} \Qab)$.}
&\cong \left[ \Warrow^{\dag}\left( (F \otimes_{\QQ} F) \otimes_{\QQ} \Qab \right) \otimes_{\ZZ} V \right]^{\Gal(F/\QQ)}
\end{align*}
We will show that this last module is isomorphic to $\Warrow^{\dag}(F \otimes_{\QQ} \Qab)\otimes_{\ZZ} V$, which will complete our proof.  Recall that $\Gal(F/\QQ)$ acts only on the right-hand factor of $F \otimes_{\QQ} F$.  In the direct sum decomposition $F \otimes_{\QQ} F \cong \bigoplus F$, we can realize the Galois action as a transitive permutation of the factors $F$.   Pick one of the factors and call it $F_e$.  Then the projection $\bigoplus F \rightarrow F_e$ induces a homomorphism 
\[
 \left[ \Warrow^{\dag}\left( (F \otimes_{\QQ} F) \otimes_{\QQ} \Qab \right) \otimes_{\ZZ} V \right]^{\Gal(F/\QQ)} \rightarrow \Warrow^{\dag}\left( F_e \otimes_{\QQ} \Qab \right) \otimes_{\ZZ} V.
\]
Identifying $F \otimes_{\QQ} F$ with the induction of $F_e$ from the trivial group to $\Gal(F/\QQ)$ then gives a map in the reverse direction 
\[
\Warrow^{\dag}\left( F_e \otimes_{\QQ} \Qab \right) \otimes_{\ZZ} V \rightarrow  \left[ \Warrow^{\dag}\left( (F \otimes_{\QQ} F) \otimes_{\QQ} \Qab \right) \otimes_{\ZZ} V \right]^{\Gal(F/\QQ)}.
\]
These two maps are inverses to each other, which completes the proof.
\end{proof}

\begin{Rmk}
The definition of $(\varphi^{-1}, \Gamma)$-module given above should be regarded as preliminary for at least two reasons.  The first reason is that the base ring $\Warrow^\dagger(\Qab) \otimes_{\ZZ} K$ is most likely too big to be the correct choice for motives other than Artin motives. An alternative construction, which accounts more carefully for the map to Fontaine's de\thinspace Rham period ring, is suggested in \cite[Definition~4.3]{Ked13b}.  The second reason, as we will see below, is that the functor from Artin motives to $(\varphi^{-1}, \Gamma)$-modules is not fully faithful.

In future work, we hope to produce a theory of \emph{global} $(\varphi^{-1}, \Gamma)$-modules which will combine the approach here with the approach in \cite{Ked13b}, and which will further replace the $p$-typical Witt vectors here with big Witt vectors, so as to treat all nonarchimedean norms simultaneously.  
Proposition~\ref{phi,Gamma invariants} below suggests that this global analogue will yield a fully faithful functor to $(\varphi^{-1},\Gamma)$-modules.  
\end{Rmk}

\begin{Thm} \label{phi,Gamma faithful}
The functor 
\[
\underline{D}: V \leadsto \left(\Warrow^{\dag}(F \otimes_{\QQ} \Qab) \otimes_{\ZZ} V\right)^{\Gal(F/\QQ)}
\]
is a faithful functor from the category of Artin motives over $\QQ$ with coefficients in $K$ to the category of $(\varphi^{-1}, \Gamma)$-modules.  
\end{Thm}

\begin{Rmk}
Example \ref{phi,Gamma not full} below shows that in general, this faithful functor is not fully faithful.
\end{Rmk}

\begin{proof}
We saw in Theorem~\ref{Artin motive theorem} that this functor $\underline{D}$ does take values in the intended category; here we check that this functor is faithful.  Let $\underline{T}$ denote the functor 
\[
\left(\Warrow^{\dag}(F \otimes_{\QQ} \Qab) \otimes_{\ZZ} K\right) \otimes_{\Warrow^{\dag}(\Qab) \otimes_{\ZZ} K} -
\]
from $\Warrow^{\dag}(\Qab) \otimes_{\ZZ} K$-modules to $\Warrow^{\dag}(F \otimes_{\QQ} \Qab) \otimes_{\ZZ} K$-modules.  To show that $\underline{D}$ is faithful, it suffices to show that the composite $\underline{T} \circ \underline{D}$ is faithful.  We saw in the proof of Theorem~\ref{Artin motive theorem} that this composition of functors was equivalent to
\[
V \leadsto \Warrow^{\dag}(F \otimes_{\QQ} \Qab) \otimes_{\ZZ} V.
\]
If $V$ is a vector space over a field of characteristic~0, then $V/\ZZ$ is flat and so the map
\[
V \rightarrow \Warrow^{\dag}(F \otimes_{\QQ} \Qab) \otimes_{\ZZ} V, \text{ given by } v \mapsto 1 \otimes v, 
\]
is injective.  Thus if $\phi_1, \phi_2: V \rightarrow V'$ are two different morphisms of Artin motives, and say $\phi_1(v) \neq \phi_2(v)$, then 
\[
\underline{T} \circ \underline{D}(\phi_1)(1 \otimes v) = 1 \otimes \phi_1(v) \neq 1 \otimes \phi_2(v) = \underline{T} \circ \underline{D}(\phi_2)(1 \otimes v).  
\]
This shows that $\underline{D}$ is faithful in the specific case that the coefficients are characteristic~0.  

Essentially the same proof works in characteristic~$p$.  If $\phi_1(v) \neq \phi_2(v)$, then we can check that $1 \otimes \phi_1(v) \neq 1 \otimes \phi_2(v)$ by checking that they have different images in  
\[
\left(\Warrow^{\dag}(F \otimes_{\QQ} \Qab)/p \Warrow^{\dag}(F \otimes_{\QQ} \Qab)\right) \otimes_{\ZZ/p\ZZ} V',
\]
for which we use that $V'$ injects into this ring, again by flatness.
\end{proof}

\begin{Eg} \label{phi,Gamma not full}
We work out a specific example which shows that in general, the functor from Theorem~\ref{phi,Gamma faithful} is not full.  Consider the case $p = 5$ and $V$ is the $\QQ$-vector space $\QQ(i)$ equipped with the obvious $\Gal(\overline{\QQ}/\QQ)$-action.  In the notation of Theorem~\ref{phi,Gamma faithful}, we may take $F = \QQ(i)$.  We claim that multiplication by $[i \otimes 1] \otimes i$ is a $(\varphi^{-1}, \Gamma)$-module endomorphism of 
\[
\left(\Warrow^{\dag}(\QQ(i) \otimes_{\QQ} \Qab) \otimes_{\ZZ} \QQ(i)\right)^{\Gal(\QQ(i)/\QQ)}
\]
which is not induced by an Artin motive endomorphism of $V = \QQ(i)$.  This will show that the functor in question is not full.  

We first show that multiplication by $[i \otimes 1] \otimes i$ does map $\left(\Warrow^{\dag}(\QQ(i) \otimes_{\QQ} \Qab) \otimes_{\ZZ} \QQ(i)\right)^{\Gal(\QQ(i)/\QQ)}$ into itself.  For varying $j$, let $r_j$ denote elements of $\Warrow^{\dag}(\QQ(i) \otimes_{\QQ} \Qab)$ and let $v_j$ denote elements of our Artin motive $V$.  If $\sum r_j \otimes v_j$ is fixed by the non-identity automorphism $\tau \in \Gal(\QQ(i)/\QQ)$, then so is $[i \otimes 1] \otimes i$, as the following computation shows:
\begin{align*}
\tau \left( [i \otimes 1] \otimes i \sum r_j \otimes v_j \right)  & = -[i \otimes 1] \otimes -i \sum \tau(r_j) \otimes \tau(v_j)\\
&= [i \otimes 1] \otimes i \sum \tau(r_j) \otimes \tau(v_j) \\
&= [i \otimes 1] \otimes i \sum r_j \otimes v_j.
\end{align*}

We next show that multiplication by $[i \otimes 1] \otimes i$ commutes with $\varphi$.  It suffices to show that the ghost components of $\varphi^{-1} \left( [i \otimes 1] w \right)$ and $[i \otimes 1] \varphi^{-1}(w)$ are equal, and this follows immediately from the fact that the ghost components of $[i \otimes 1]$ are 
\[
\left( i \otimes 1, (i \otimes 1)^5, (i \otimes 1)^{25}, \ldots\right) = (i \otimes 1, i \otimes 1, i \otimes 1, \ldots).
\]

To complete the proof that $[i \otimes 1] \otimes i$ is $(\varphi^{-1}, \Gamma)$-module endomorphism of 
\[
\left(\Warrow^{\dag}(\QQ(i) \otimes_{\QQ} \Qab) \otimes_{\ZZ} \QQ(i)\right)^{\Gal(\QQ(i)/\QQ)},
\]
we should check $\Gamma$-equivariance, but this is obvious because $[i \otimes 1] \otimes i$ is $1$ in the only coordinate on which $\Gamma$ acts.

We now have a $(\varphi^{-1}, \Gamma)$-module endomorphism, and we will show that it is not induced by an Artin motive endomorphism of $V$.  Consider its action 
\[
[1 \otimes 1] \otimes 1 \mapsto [i \otimes 1] \otimes i.
\]
We claim that no Artin motive endomorphism of $V$ can have this effect.  It is easy to see that an Artin motive endomorphism of $V = \QQ(i)$ must have the form $r + si \mapsto ar + bsi$ for some scalars $a,b \in \QQ$.  Thus we are reduced to showing that
\[
[i \otimes 1] \otimes i \neq [1 \otimes 1] \otimes a
\]
for any choice of scalar $a \in \QQ$.  By projecting to the first Witt component, we would like to show that 
\[
i \otimes 1 \otimes i \neq 1 \otimes 1 \otimes a \in \QQ(i) \otimes_{\QQ} \Qab \otimes_{\ZZ} \QQ(i).
\]
For example, these elements have different images in $\QQ(i) \otimes_{\QQ} \Qab \otimes_{\QQ} \QQ(i)$.  
\end{Eg}

\begin{Rmk}
The problem that was exploited in the preceding example is that, in general,
\[
\left(\Warrow^{\dag}(F \otimes_{\QQ} \Qab)\right)^{\varphi^{-1}, \Gamma} \neq \ZZ.
\]
In the above example, $[i \otimes 1]$ was a $(\varphi^{-1}, \Gamma)$-invariant.  The following proposition identifies these invariants explicitly.
\end{Rmk}

\begin{Prop} \label{phi,Gamma invariants}
Let $F_0$ denote the maximal Galois subextension of $F/\QQ$ in which $p$ splits completely.  Let $\mathcal{O}_{F_0}$ denote its ring of integers, and let $\mathcal{O}_{F_0, (p)}$ denote its localization at $(p)$.  Then 
\[
\left(\Warrow^{\dag}(F \otimes_{\QQ} \Qab)\right)^{\varphi^{-1}, \Gamma} \cong \mathcal{O}_{F_0, (p)}.
\]
\end{Prop}

\begin{proof}
The ghost components of an invariant must all be equal to some single element in $F \otimes_{\QQ} \Qab$ (by $\varphi^{-1}$-invariance).  The $\Gamma$-invariants of $F \otimes_{\QQ} \Qab$ are precisely $F \otimes_{\QQ} \QQ \cong F$ (to see this, choose a $\QQ$-basis of $F$ and use the corresponding isomorphism $F \otimes_{\QQ} \Qab \cong \oplus \Qab$).  Thus we would like to determine, for which elements $f \in F$, is $(f \otimes 1,f \otimes 1,\ldots)$ in the image of $\Warrow^{\dag}(F \otimes_{\QQ} \Qab)$ under the ghost map. 

Assume $f \in F$ is such that $(f, f, \ldots)$ is the image of some element $\underline{x}$ in $\Warrow^{\dag}(F)$ under the ghost map. 
Choose $b \geq 1$ for which $\underline{x} \in \Warrow^b(F)$;
by Lemma~\ref{effect of inverse Frobenius}, we also have $\underline{x} \in \Warrow^{b'}(F)$ for $b' = \max\{b/p, 1\}$. It follows that
$\underline{x} \in \Warrow^1(F)$.
Let $\widehat{F}$ denote the completion of $F$ at some fixed prime above $p$.  We have
\[
\Warrow^{1}(F) \subseteq \Warrow^{1}(\widehat{F}) \subseteq \Warrow^{1}(\Cp) \cong \Warrow^{1}(\varprojlim_{x \mapsto x^p} \Cp),
\]
where the last isomorphism is a consequence of Theorem~\ref{compare to positive characteristic1}.  Abbreviate $\varprojlim \Cp$ by $\tilde{E}$.  Because $\tilde{E}$ is a perfect field of characteristic $p$, by Theorem~\ref{overconvergence in characteristic p}, we have a further isomorphism $\Warrow^{1}(\tilde{E}) \cong W^{1}(\tilde{E})$.  The above maps are all $\varphi^{-1}$-compatible, and so we deduce that our element $\underline{x} \in \left( \Warrow^{\dag}(F) \right)^{\varphi^{-1}}$ corresponds to an element in $W(\tilde{E})^{\varphi^{-1}} = W(\tilde{E}^{\varphi^{-1}}) = W(\FF_p).$  

Assume $F$ is a number field.  The field $F_0$ (as in the statement of the proposition) is the largest subfield of $F$ with the property that, for every prime $\mathfrak{p}$ above $p$, the image of $F_0$ in the completion $F_{\mathfrak{p}}$ is equal to $\QQ_p$.  The previous paragraph shows that our invariant element $f$ is an element of $F_0$.  On the other hand, it further has the property that its image is in $\ZZ_p$ for every prime above $p$, so in fact the element $f$ is in $\mathcal{O}_{F_0, (p)}$.  

Conversely, we would like to show that for any $f \in \mathcal{O}_{F_0, (p)}$, the element $(f \otimes 1, f \otimes 1, \ldots)$ is in the image of $\left(\Warrow^{\dag}(F \otimes_{\QQ} \Qab)\right)^{\varphi^{-1}, \Gamma}$ under the ghost map.  Clearly an element of $\Warrow^{\dag}(F \otimes_{\QQ} \Qab)$ with the above ghost components is invariant under $\varphi^{-1}$ and $\Gamma$, so it suffices to check that the element $(f \otimes 1, f \otimes 1, \ldots)$ is in the image of $\Warrow^{\dag}(F \otimes_{\QQ} \Qab)$ under the ghost map.  

For $f \in \mathcal{O}_{F_0, (p)}$, we can find a unique element $\underline{x} \in \Warrow(F)$ with ghost components $(f,f,\ldots)$; we will show that $\underline{x}$ is in fact an element of $\Warrow^{\dag}(F)$.  We know that for every completion $\widehat{F}$ of $F$ at a place above $p$, the image of $\underline{x}$ in $\Warrow(\widehat{F})$ lies in $\Warrow(\mathcal{O}_{\widehat{F}}) \subseteq \Warrow^{\dag}(\widehat{F})$.  Because our norm on $F$ corresponds to the supremum of the $p$-adic norms on the various completions $\widehat{F}$, we deduce that $\underline{x} \in \Warrow^{\dag}(F)$, as desired.
\end{proof}

\begin{Rmk}
In future work, we will construct an analogue $\WWarrow^{\dag}(F \otimes_{\QQ} \Qab)$ of $\Warrow^{\dag}(F \otimes_{\QQ} \Qab)$ which uses big Witt vectors instead of $p$-typical Witt vectors, and which uses as norm the supremum over (at least) all finite places, instead of the supremum over all places above $p$.   (The map $\varphi^{-1}$ in this case should also be replaced by a different Frobenius map for each prime.)  For this more global base ring, the analogue of Proposition~\ref{phi,Gamma invariants} should be 
\[
\left(\WWarrow^{\dag}(F \otimes_{\QQ} \Qab)\right)^{\varphi^{-1}, \Gamma} \cong \mathcal{O}_{L},
\] 
where $L$ is the maximal subextension of $F$ in which \emph{every} prime $p$ splits completely.  In other words, the $(\varphi^{-1}, \Gamma)$-invariants should be $\ZZ$.  It should then follow that the functor 
\[
V \leadsto \left(\WWarrow^{\dag}(F \otimes_{\QQ} \Qab) \otimes_{\ZZ} V\right)^{\Gal(F/\QQ)}
\]
is fully faithful.
\end{Rmk}

\bibliography{padicHodge3}
\bibliographystyle{plain} 

\end{document}